\documentclass[a4paper,10pt]{article}

\usepackage{authblk}
\usepackage{amsmath,amssymb,amsthm}
\usepackage[T1]{fontenc}
\usepackage[utf8]{inputenc}

\usepackage{pdflscape}

\usepackage{ifpdf}
\ifpdf
   \usepackage[pdftex,bookmarks=true,bookmarksnumbered,%
    plainpages=false,hypertexnames=false,pdfpagelabels,%
    colorlinks=true,linkcolor=black,urlcolor=blue,citecolor=red,%
    pdfstartview=FitH]{hyperref}
   \usepackage[pdftex]{graphicx,color}
\else
    \usepackage[dvips,bookmarks=true,bookmarksnumbered,plainpages=false,%
     colorlinks=true,linkcolor=black,urlcolor=blue,citecolor=red,%
     hypertexnames=false,pdfstartview=FitH,breaklinks=true]{hyperref}
    \usepackage{pstricks,pst-node,pst-text,pst-3d}
    \usepackage[dvips]{graphicx}
\fi

\newtheorem{theorem}{Theorem}
\newtheorem{prop}[theorem]{Proposition}

\theoremstyle{definition}

\theoremstyle{remark}
\newtheorem{rem}{Remark}

\usepackage{enumerate}
\usepackage{bbm}
\title{An unbiased Monte Carlo estimator for derivatives. Application to CIR.}
\author[1]{Victor Reutenauer}
\affil[1]{Fotonower, France.}
\author[2]{Etienne Tanr\'e}
\affil[2]{Universit\'e C\^ote d'Azur, Inria, France.}

\begin{document}
	
\maketitle

\begin{abstract}
	In this paper, we present extensions of the exact simulation algorithm introduced by Beskos et al.~\cite{BPR06}. First, a  modification in the order 
	in which the simulation is done accelerates the algorithm. 
	 In addition, we propose a truncated version of the modified algorithm.  
	 We obtain a control of the bias of this last version, exponentially small in function of the truncation parameter.
	Then, we extend it to more general drift functions.
	Our main result is an unbiased algorithm to approximate 
	the two first derivatives with respect to the initial condition \(x\) of quantities with the form \(\mathbb{E}\Psi(X_T^x)\). 
	We describe it in details in dimension 1 and also discuss its multi-dimensional extensions for the evaluation of \(\mathbb{E}\Psi(X_T^x)\).
	Finally, we apply the algorithm to the CIR process and perform numerical tests to compare it with classical approximation procedures.
%	\ladim{je me demande s'il ne faut finalement pas parler de la dim dans l'abstract \tored{VR 1-12-16 : OUI}}
%	We provide limits and possibility of using this methodology in higher dimension, including the computation of hedges.
\end{abstract}

\noindent\textbf{Keywords:} Unbiased Monte Carlo methods; Monte Carlo Approximation of Derivatives; Exact Simulation of SDE.

\noindent\textbf{AMS2010 class:} 65C05, 60J60.

\section{Introduction \label{sec:intro}}

In this paper, we are interested in the approximation of the law of 
a one dimensional stochastic process \((X^x_t, t\geq 0)\), defined as the unique solution of a 
Stochastic Differential Equation (SDE)
\begin{equation}\label{eq:edsgene}
X_T^x = x + \int_0^T\alpha(X_t^x)dt + \int_0^T\sigma (X_t^x)dW_t,
\end{equation}
with smooth coefficients \(\alpha\) and \(\sigma\). 
Let \(\Psi\) be a measurable function. The quantities we aim to evaluate
take form
\begin{equation}
\label{eq:prix}
P_\Psi(x):=\mathbb{E}\Psi(X_T^x).
\end{equation}
We also evaluate their sensitivities to the parameters of the model.
We are especially interested in the dependance on the initial condition \(x\), 
\begin{align}
\label{eq:delta}
\Delta_\Psi(x) &:= \frac{d}{dx} \mathbb{E}\Psi(X_T^x)\\
\label{eq:gamma}
\Gamma_\Psi(x) &:= \frac{d^2}{dx^2} \mathbb{E}\Psi(X_T^x).
\end{align}
These two derivatives are 
known as Delta and Gamma in the context of financial mathematics.

The most simple method to approximate \eqref{eq:prix} consists in
a time discretisation (say with step \(\delta\)) of \eqref{eq:edsgene} 
with an Euler scheme.
For an approximation of \eqref{eq:delta} or \eqref{eq:gamma}, we should 
evaluate \eqref{eq:prix} with two or three values, say for 
\(x-dx\), \(x\) and \(x+dx\).
Then, we use a finite difference approximation of the derivatives. This 
method is very simple to implement, but we have three sources of error:
\begin{enumerate}[1)]
	\item two biases due to 
	\begin{enumerate}[a -]
		\item the time discretisation \(\delta\);
		\item the finite difference approximation parameter \(dx\);
	\end{enumerate}
	\item the statistical error.
\end{enumerate}
In \cite{BR05} and \cite{BPR06}, the authors 
 proposed an exact simulation algorithm for one dimensional SDE
 with constant diffusion coefficient \(\sigma(x)\equiv 1\) (see Section~\ref{subsec:21}).
 This method removes the bias of type \textit{a} in the approximation
 of \(P_\Psi(x)\). More recently, several authors have worked on
 algorithms withous bias of type \textit{a}. For instance, in Bally and 
 Kohatsu-Higa~\cite{BKH15}, a theoretical stochastic representation of the 
 parametrix  method is developped and used successfully to reach this goal.
 Similar ideas are developped in \cite{HLTT16} to evaluate \(P_\Psi\) 
 for smooth functions \(\Psi\) and diffusion process in \(\mathbb{R}^d\).
 Gobet and M'rad~\cite{gobetmrad16} have proposed a multilevel Monte Carlo 
 method with random number of levels. They succeed to avoid bias for 
 Lipschitz continuous function \(\Psi\).

Otherwise, in \cite{fournier_etal_1999} the Malliavin calculus theory 
is developed to obtain expressions of the derivatives \(\Delta_\Psi(x)\) 
and \(\Gamma_\Psi(x)\) without bias of type b-. The authors write
\[
\dfrac{d}{dx}\mathbb{E}[\Psi(X_T^x)] = \mathbb{E}[\Psi(X_T^x)H_T],
\]
where \(H_T\) is an explicit random weight. 

In this paper, we 
extend 
Beskos et al. method of simulation:  
we simulate the Poisson process by ordering 
the points in increasing ordinate (see Sec.~\ref{subsec:pp_theo_simulation}).
With this modification, the rejection of Brownian bridge trajectories are 
decided faster and the efficiency of the algorithm is higher. Moreover, one 
should relax a little bit the assumption on the drift coefficient \(\alpha\). 
Our algorithm is efficient to compute \(P_\Psi\) even for full path dependent function \(\Psi\). Thereby, even if it applies essentially for one dimensional diffusion processes, our work, in this setting, is more general than \cite{HLTT16,gobetmrad16}.

Furthermore, we propose an unbiased algorithm to compute the derivatives 
\eqref{eq:delta} and \eqref{eq:gamma}.
The idea combines Fourni\'e et al.~\cite{fournier_etal_1999} formula and some 
generalisation of Beskos et al.~\cite{BPR06} rejection procedure.

The paper is organised as follows. 
We describe the algorithms in a general context
in Section~\ref{sec:section2}. Section~\ref{sec:cirmodel} is
devoted to a detailed presentation for the CIR model.
We compare the efficiency of our algorithm 
with classical estimators in Section~\ref{sec:numerical}.

	\thanks{\textbf{Acknowledgment: }The authors would like to thank gratefully CA-CIB and Inria. This work started during an official collaboration
		between their teams. They are grateful to Inria Sophia Antipolis - Méditerranée
		``Nef'' computation cluster for providing resources and support. They want also to thank Pierre \'Etor\'e for pointing out a mistake in a preliminary version of this work.}

\section{Unbiased Estimators}\label{sec:section2}
\subsection{Beskos, Papaspiliopoulos and Roberts
	unbiased estimator}\label{subsec:21}
Here, we recall the main ideas developped in \cite{BR05,BPR06}
to exactly simulate  the solution of
one dimensional stochastic differential equations.
Assume that the process \(X^x\) solves the equation
\begin{equation}\label{eq:edssigma1}
X_T^x = x+\int_0^T\alpha(X_t^x)dt + W_T
\end{equation}
(i.e. \(\sigma \equiv 1\) in \eqref{eq:edsgene}). 
The main idea 
is a smart use of
Girsanov Theorem:
\begin{equation*}
\mathbb{E}\left[\Psi(X_T^x) 
\vphantom{\exp\left(\int_0^T\alpha(B_s) dB_s - \int_0^T\dfrac{\alpha^2(B_s)}{2}ds\right)}
\right] = \mathbb{E}\left[\Psi(B_T^x)
\exp\left(\int_0^T\alpha(B_t^x) dB_t^x - \int_0^T\dfrac{\alpha^2(B_t^x)}{2}dt\right)\right]
\end{equation*}
where \((B_t^x)_{t\geq 0}\) is a one dimensional Brownian motion with 
\(B_0^x = x\). 
The dimension allows one to transform the stochastic integral:
\begin{equation}\label{eq:removstochint}
\int_0^T\alpha(B_t^x) dB_t^x = A(B_T^x) - A(B_0^x) - \int_0^T\dfrac{\alpha^\prime(B_t^x)}{2}dt,
\end{equation}
where \(A(x) = \int_0^x \alpha(y) dy\). Then, one obtains
\begin{equation}\label{eq:calculbpr}
\mathbb{E}\left[\vphantom{\exp\left(-\int_0^T\varphi(B_t^x)dt\right)}\Psi(X_T^x) \right] = 
\mathbb{E}\left[\Psi(B_T^x)\exp\bigg(\!A(B_T^x) - A(B_0^x)-\!\!\int_0^T\!\!\frac{(\alpha^2+\alpha^\prime)}{2}(B_t^x)dt\bigg)\right].
\end{equation}
Next, we replace in \eqref{eq:calculbpr} the Brownian motion \((B^x_t,0\leq t\leq T)\) by a Brownian bridge \((\tilde{B}^x_t,0\leq t\leq T)\), where the final value \(\tilde{B}^x_T\) has the distribution
\begin{equation}\label{eq:distribfinBBrid}
\mathbb{P}(\tilde{B}^x_T\in d\theta)  = C\exp\left(- \dfrac{(\theta-x)^2}{2T}+
A(\theta) \right)d\theta,
\end{equation}
where \(C\) is a normalisation.
Then, denote 
\begin{equation}\label{eq:defdephi}
\varphi(y) = \frac{\alpha^2(y) + \alpha^\prime(y)}{2},
\end{equation} 
there is a constant \(\tilde{C}\), depending on \(C\) and \(\alpha\) (but not on \(\Psi\)), such that
\begin{equation}\label{eq:girsanovdansbpr}
\mathbb{E}\left[\Psi(X_T^x) \right] = \tilde{C}\mathbb{E}\left[\Psi(\tilde{B}^x_T)
\exp\left(-\int_0^T\varphi(\tilde{B}^x_t)dt\right)\right].
\end{equation}
If we moreover assume that \(\varphi\) takes value in a compact set, say \(0 \leq \varphi(y) \leq K\), one can exactly simulate the diffusion \(X^x\) with a 
rejection procedure. Namely, one simulates a path of the Brownian bridge \(\tilde{B}^x\) 
and accept it with probability \(\exp\left(-\int_0^T\varphi(\tilde{B}^x_t)dt\right)\). 
To do it, one simulates a Poisson process (independent of \(\tilde{B}^x\)) of unit intensity on \([0,T]\times[0,K]\) and accepts the Brownian bridge path if and only if there is no point of the Poisson process
in the hypograph \(D(\omega)\) of \(\varphi(\tilde{B}^x_t)\)
\begin{equation}\label{eq:defhypograph}
D(\omega) = \left\{(t,y)\in[0,T]\times[0,K], y\leq\varphi(\tilde{B}^x_t)\right\}.
\end{equation}
It is easy to verify that the probability to accept the path is 
\(\exp\left(-\int_0^T\varphi(\tilde{B}^x_t)dt\right)\). Furthermore, 
we only need to know the value of the Brownian bridge at a finite number
of times \(0<t_1<\cdots<t_n \leq T\), the abscissas of the points of the Poisson process.
So, we have
\begin{equation}\label{eq:egalloiavecpoissonhypo}
	\mathbb{E}\left[\Psi(X_T^x) \right] = \mathbb{E}\left[\Psi(\tilde{B}^x_T)\middle| N\cap D(\omega) = \emptyset\right],
	\end{equation}
	where \(N\) is a Poisson process with unit intensity on \([0,T]\times[0,K]\), independent of \((\tilde{B}^x_t,0\leq t\leq T)\).
\begin{rem}\label{rem:relachehypsurphi}
	We have written this short presentation under the assumption \(0\leq \varphi \leq K\). It should be easily generalised to
	the cases where:
	\begin{enumerate}
		\item \(\varphi\) is bounded, but not necessary nonnegative. 
		In this case, we only have to replace in \eqref{eq:girsanovdansbpr} the function \(\varphi\) by \(\varphi - \inf_\mathbb{R}\varphi\) and the constant \(\tilde{C}\) by \(\tilde{C}\exp(-T\inf_{\mathbb{R}}\varphi)\).
		\item \label{extensionavecmin}\(\varphi\) has no finite global upper bound, but has an upper bound in \(+\infty\) or \(-\infty\). 
		For instance, \(\limsup_{y\rightarrow  - \infty}\varphi(y) = +\infty\) 
		and \(\limsup_{y\rightarrow  + \infty}\varphi(y) < \infty\). Here, we only have to first simulate the infimum \(m(\omega)\) of \(\tilde{B}^x\) on \([0,T]\) and the time 
		\(t_m(\omega)\)
		at which it is reached. Then, we simulate a Poisson process on \([0,T]\times [0,\tilde{K}(\omega)]\) with 
		\(\tilde{K}(\omega) = \sup_{y\geq m(\omega)}\varphi(y)\) (see \cite{BPR06}).
	\end{enumerate}
\end{rem}
Williams decomposition of Brownian paths \cite{williams} 
gives the conditional law \((\tilde{B}^x_t,0\leq t \leq T | m(\omega), t_m(\omega))\): conditionally to \(m\) and \(t_m\), 
the processes \((\tilde{B}^x_{t_m + t} - m,0\leq t \leq T - t_m)\) and 
\((\tilde{B}^x_{t_m - t} - m ,0\leq t \leq t_m)\) are two independent 
Bessel bridges processes of dimension 3. Such a process is simple to exactly 
simulate at a finite number of times.

\subsection{Unbiased estimator of the first derivative (Delta)}\label{subsec:22}
In this section, we present our main results. We generalise the unbiased algorithm introduced by Beskos et al. 
\cite{BPR06} to approximate the sensitivities \(\frac{d}{dx}\mathbb{E}[\Psi(X_T^x)]
\)
with an unbiased estimator.
\begin{prop}\label{prop:estimateursansbiaisdelta}
	Let \((X^x_t,0\leq t\leq T)\) be the solution of \eqref{eq:edssigma1}, starting from \(x\), and \(\Psi\) a measurable function. 
	Assume that \(\forall y\in \mathbb{R}\), we have \(-\hat{K}\leq \alpha^\prime(y)\leq 0\) and 
	\(0\leq \alpha^2(y)+\alpha^\prime(y)\leq 2K\).
	Then, an unbiased Monte Carlo procedure to evaluate \(\frac{d}{dx} \mathbb{E}\Psi(X^x_T)\) is available 
		\begin{multline*}
		\frac{d}{dx} \mathbb{E}\Psi(X^x_T) = 
		- \mathbb{E}\left[\dfrac{x\Psi(\tilde{B}_T^x)}{T} \middle| N\cap D = \emptyset\right]\\
		+\mathbb{E}\left[\dfrac{\Psi(\tilde{B}_T^x)}{T}\left(\tilde{B}_T^x-T \alpha(\tilde{B}_{U_2T}^x)\right)
		\mathbbm{1}_{\hat{N}\cap\hat{D}=\emptyset}
		\middle| N\cap D = \emptyset\right]\\
		-\mathbb{E}\left[
		\Psi(\tilde{B}_T^x)
		\left(\tilde{B}_{U_1T}^x - U_1T\alpha(\tilde{B}_{U_1U_2T}^x)\right)
		\alpha^\prime(\tilde{B}_{U_1T}^x)
		\mathbbm{1}_{\hat{N}\cap\hat{D}^1=\emptyset}
		\middle| N\cap D = \emptyset\right],
		\end{multline*}	
	where:
	\begin{itemize}
%		\item \(\varphi = (\alpha^2+\alpha^\prime)/2\);
		\item \((\tilde{B}_t^x, 0\leq t\leq T)\) is a Brownian bridge with \(\tilde{B}^x_T\) given by \eqref{eq:distribfinBBrid};
		\item \(D\) is the hypograph  of \(\varphi(\tilde{B}_t^x)\) (see \eqref{eq:defdephi} and \eqref{eq:defhypograph});
		\item \(\hat{D}\) is the hypograph of \(-\alpha^\prime(\tilde{B}_t^x)\) and \(\hat{D}_1 = \hat{D}\cap ([0,U_1T]\times \mathbb{R}_+)\);
		\item \(N\) and \(\hat{N}\) are two independent Poisson processes with
		unit intensity on \([0,T]\times[0,K]\) and \([0,T]\times[0,\hat{K}]\), (independent of \(\tilde{B}^x\));
		\item \(U_1\) and \(U_2\) are two independent random variables with uniform distribution on \([0,1]\) (independent of \(\tilde{B}^x\), \(N\) and \(\hat{N}\)).
	\end{itemize}
	
\end{prop}
We first recall basic results on Malliavin calculus (see Fourni\'e et al. \cite{fournier_etal_1999}) useful to detail our algorithm.
The process \((X^x_t, t\geq 0)\) is the unique solution of \eqref{eq:edssigma1} 
with \(X_0^x =x\).
We denote by \((Y^x_t,t\geq 0)\) the associated first variation process 
\[
Y^x_t := \dfrac{d}{dx} X^x_t.
\]
It solves the linear SDE
\begin{equation*}
\begin{aligned}
dY^x_t & = Y^x_t\alpha^\prime(X^x_t) dt\\
Y^x_0 & = 1.
\end{aligned}
\end{equation*}
The solution 
is
\begin{equation}\label{eq deriveeflot}
Y^x_t =  \exp\left(\int_0^t \alpha^\prime (X^x_s) ds\right).
\end{equation}
Furthermore, it is known that the Malliavin derivative \(D_t X^x_T\) satisfies
\begin{equation}
\begin{aligned}
\forall t\leq s,\quad dD_tX^x_s & = D_tX^x_s\alpha^\prime(X^x_s)ds\\
D_tX^x_t &= 1.
\end{aligned}
\end{equation}
We deduce that \(Y^x\) and \(D_tX^x\) are linked by the identity
	\begin{equation}\label{eq:flotmalliavin}
	D_tX^x_T = \dfrac{Y^x_T}{Y^x_t}.
	\end{equation}
So,
\begin{equation*}
Y^x_T = Y^x_t D_tX^x_T  = \int_0^T a(t)Y^x_t D_tX^x_Tdt,
\end{equation*}
where \(a\) is any \(L^2\) function such that \(\int_0^Ta(t)dt = 1\).
For instance, we use in this paper \(a(t)\equiv \frac{1}{T}\).

Following Fourni\'e et al. \cite{fournier_etal_1999}, and using classical results on Malliavin calculus (integration by parts formula, see~\cite{MR2200233}),
we obtain for \(\Psi\in C^1\)
\begin{align}
\nonumber\frac{d}{dx} \mathbb{E}\Psi(X^x_T) & =
\mathbb{E}
\left[
\Psi^\prime(X^x_T)Y^x_T
\right]\\
\nonumber& = \dfrac{1}{T}\mathbb{E} \left[\int_0^T\Psi^\prime(X^x_T)D_t(X^x_T) Y^x_t dt\right] = \dfrac{1}{T}\mathbb{E} \left[\int_0^TD_t(\Psi(X^x_T)) Y^x_t dt\right]\\
\nonumber& = \dfrac{1}{T}\mathbb{E} \left[\Psi(X^x_T) \delta(Y^x_t)\right] \\
\label{eq:expressiongenedelta}&= \dfrac{1}{T}\mathbb{E} \left[\Psi(X^x_T) \int_0^T Y^x_t dW_t\right].
\end{align}
\begin{rem}
This last identity remains true if \(\Psi\) is not a smooth function (see \cite{MR2200233}).
\end{rem}
After this short remind on Malliavin calculus theory, we now prove Proposition~\ref{prop:estimateursansbiaisdelta}.
\begin{proof}[Proof of Proposition~\ref{prop:estimateursansbiaisdelta}]
We use the one dimension setting to remove the stochastic integral in \eqref{eq:expressiongenedelta}
\begin{align}
\nonumber\int_0^T Y^x_t dW_t & = W_TY^x_T - W_0Y^x_0 - \int_0^T W_t dY^x_t \\
\label{eq:intydw}
& = W_TY^x_T - W_0Y^x_0 - \int_0^T W_t Y^x_t \alpha^\prime(X^x_t)dt.
\end{align}
The evaluation of the integral in the last term would introduce a bias.
To avoid it, one uses a classical identity. Namely, consider a stochastic process \((\gamma_t,0\leq t\leq T)\), we have
\begin{equation}\label{eq integunif}
\int_0^T \gamma_t dt = T\bar{\mathbb{E}}(\gamma_{UT})
\end{equation}
where \(U\) is random variable with uniform distribution on \([0,1]\), independent of \(\gamma\) and
\(\bar{\mathbb{E}}\) denotes the expectation with respect to \(U\).
	The drawback of the last expression is the
	increase of the variance. See \cite{mairetanre} for a discussion on this topic.
Using this property and \eqref{eq deriveeflot}, we obtain
\begin{multline*}
\frac{d}{dx} \mathbb{E}\Psi(X^x_T)  = \mathbb{E}\left[\dfrac{\Psi(X^x_T)}{T}\left(W_T\exp\int_0^T\alpha^\prime(X^x_s)ds
- W_0\right.\right.\\
\left.\left.- T W_{U_1T}\alpha^\prime(X^x_{U_1T})\exp\int_0^{U_1T}\alpha^\prime(X^x_s)ds\right)\right],
\end{multline*}
where \(U_1\) is a random variable independent of \((X^x_t,t\in[0,T])\) with uniform law on \([0,1]\).
As in Section~\ref{subsec:21}, we finally apply Girsanov Theorem 
\begin{multline*}
\frac{d}{dx} \mathbb{E}\Psi(X^x_T) = \tilde{C} \mathbb{E}\left[\dfrac{\Psi(\tilde{B}^x_T)}{T}\exp\left(-\dfrac{1}{2}\int_0^T\alpha^2(\tilde{B}^x_s) + \alpha^\prime(\tilde{B}^x_s)ds\right)\right.\\
\times\left(\left(\tilde{B}^x_T-\int_0^T\alpha(\tilde{B}^x_t)dt\right)
\exp\left(\int_0^T\alpha^\prime(\tilde{B}^x_s)ds\right)
- x\right.\\
- \left.\left. T\left(\tilde{B}^x_{U_1T} - \int_0^{U_1T}\alpha(\tilde{B}^x_s)ds\right)\alpha^\prime(\tilde{B}^x_{U_1T})\exp\int_0^{U_1T}\alpha^\prime(\tilde{B}^x_s)ds\right)\right],
\end{multline*}
where \((\tilde{B}^x_t,0\leq t\leq T)\) is a Brownian bridge with final distribution
given by \eqref{eq:distribfinBBrid} and \(U_1\sim\mathcal{U}(0,1)\) is independent of \(\tilde{B}^x\).
We use the same rejection procedure as in Section~\ref{subsec:21} to obtain
	\begin{multline*}
	\frac{d}{dx} \mathbb{E}\Psi(X^x_T) =  \mathbb{E}\left[\dfrac{\Psi(\tilde{B}^x_T)}{T}
	\times\left((\tilde{B}^x_T-\int_0^T\alpha(\tilde{B}^x_t)dt)
	\exp\left(\int_0^T\alpha^\prime(\tilde{B}^x_s)ds\right)
	- x\right.\right.\\
	- \left.\left. T\left(\tilde{B}^x_{U_1T} - \int_0^{U_1T}\alpha(\tilde{B}^x_s)ds\right)\alpha^\prime(\tilde{B}^x_{U_1T})\exp\int_0^{U_1T}\alpha^\prime(\tilde{B}^x_s)ds\right)\middle|N\cap D(\omega) = \emptyset\right].
	\end{multline*}
	\begin{multline*}
	\frac{d}{dx} \mathbb{E}\Psi(X^x_T) = 
	- \mathbb{E}\left[\dfrac{x\Psi(\tilde{B}^x_T)}{T} \middle| N\cap D(\omega) = \emptyset\right]\\
	+\mathbb{E}\left[\dfrac{\Psi(\tilde{B}^x_T)}{T}\left(\tilde{B}^x_T-T \alpha(\tilde{B}^x_{U_2T})\right)
	\exp\left(\int_0^T \alpha^\prime(\tilde{B}^x_s)ds\right) \middle| N\cap D(\omega) = \emptyset\right]\\
	-\mathbb{E}\left[
	\Psi(\tilde{B}^x_T)
	\left(\tilde{B}^x_{U_1T} - U_1T\alpha(\tilde{B}^x_{U_1U_2T})\right)
	\alpha^\prime(\tilde{B}^x_{U_1T})
	\right.\\
	\left.\exp\left(\int_0^{U_1T}\alpha^\prime(\tilde{B}^x_s)ds\right)
	 \middle| N\cap D(\omega) = \emptyset\right],
	\end{multline*}	
	where \(U_2\) is a random variable with uniform distribution on \([0,1]\),
	independent of \(\tilde{B}^x\), \(N\) and \(U_1\).

	It remains to remark that one again 
	interprets the term \label{pagepourlexpo} 
	\(
	\exp\left(\int_0^T \alpha^\prime(\tilde{B}^x_s)ds\right)\) 
	as the probability for a Poisson process to have no point in a domain.
	More precisely, we consider a Poisson process \(\hat{N}\) with unit intensity on \([0,T]\times[0,\hat{K}]\) (where \(\hat{K} = -\inf_{\mathbb{R}} \alpha^\prime\)), independent of 
	\(\tilde{B}^x\), \(N\), \(U_1\) and \(U_2\). We denote by \(\hat{D}\) the hypograph of \(-\alpha^\prime\).
%	\[
%		\hat{D}(\omega) = \left\{(t,y)\in[0,T]\times\mathbb{R}^+, y\leq - \alpha^\prime(\tilde{B}^x_t)\right\}
%	\]
	and by \(\hat{D}^1\) its restriction to \([0,U_1T]\times \mathbb{R}^+\).
	We finally have the unbiased estimator
		\begin{multline}\label{eq:deltaavecpoisson}
		\frac{d}{dx} \mathbb{E}\Psi(X^x_T) = 
		- \mathbb{E}\left[\dfrac{x\Psi(\tilde{B}^x_T)}{T} \middle| N\cap D(\omega) = \emptyset\right]\\
		+\mathbb{E}\left[\dfrac{\Psi(\tilde{B}^x_T)}{T}\left(\tilde{B}^x_T-T \alpha(\tilde{B}^x_{U_2T})\right)
		\mathbbm{1}_{\hat{N}\cap\hat{D}=\emptyset}
		\middle| N\cap D(\omega) = \emptyset\right]\\
		-\mathbb{E}\left[
			\Psi(\tilde{B}^x_T)
		\left(\tilde{B}^x_{U_1T} - U_1T\alpha(\tilde{B}^x_{U_1U_2T})\right)
		\alpha^\prime(\tilde{B}^x_{U_1T})
			\mathbbm{1}_{\hat{N}\cap\hat{D}^1=\emptyset}
		\middle| N\cap D(\omega) = \emptyset\right].
		\end{multline}	
\end{proof}

\begin{rem}
	Similarly to Remark~\ref{rem:relachehypsurphi}, we can generalise the previous estimator to function \(\varphi\) with a global lower bound and an upper bound only in one side.
	
	Furthermore, the same extension should be obtained if \(-\alpha^\prime\) has a global lower bound. In this case, we replace 
	\(-\alpha^\prime\) by \(-\alpha^\prime + \sup_{\mathbb{R}}(\alpha^\prime)\) in the definition of \(\hat{K}\), \(\hat{D}\) and \(\hat{D}^1\). We also replace \(\mathbbm{1}_{\hat{N}\cap\hat{D}=\emptyset}\) by \(\exp(T\sup_{\mathbb{R}}(\alpha^\prime))\mathbbm{1}_{\hat{N}\cap\hat{D}=\emptyset}\) and \(\mathbbm{1}_{\hat{N}\cap\hat{D}^1=\emptyset}\) by \(\exp(U_1T\sup_{\mathbb{R}}(\alpha^\prime))\mathbbm{1}_{\hat{N}\cap\hat{D}^1=\emptyset}\).
	
Our unbiased estimator can be extended if \(-\alpha^\prime\) has only a local upper bound in the same side as \(\varphi\) (i.e. \(\limsup_{y\rightarrow+\infty}\varphi(y)\) and \(\limsup_{y\rightarrow+\infty}(-\alpha^\prime(y))\) are both finite or \(\limsup_{y\rightarrow-\infty}\varphi(y)\) and \(\limsup_{y\rightarrow-\infty}(-\alpha^\prime(y))\) are both finite).
\end{rem}

\subsection{Unbiased estimator of the second derivative (Gamma)}\label{subsec:23}
In this part, we detail
an unbiased estimator of the second derivative \(\frac{d^2}{dx^2} \mathbb{E}\Psi(X^x_T)\).
We denote by \(Z^x_t\) the second variation process associated to \(X^x_t\)
\[
Z^x_t = \dfrac{d^2}{dx^2}X^x_t.
\]
It satisfies the linear stochastic differential equation
\[
Z^x_T  = \int_0^T \alpha^{\prime\prime}(X^x_s) \left(Y^x_s\right)^2
+ \alpha^\prime(X^x_s) Z^x_sds.
\]
The solution is
\begin{equation}
	Z^{x}_T =  Y^x_T \displaystyle\int_0^T \alpha^{\prime \prime}(X^x_s)Y^x_s ds.
\end{equation}
We also need the Malliavin derivative of the first variation process \(Y^x\). It satisfies
\[
D_tY^x_T = \int_0^T \alpha^{\prime\prime}(X^x_s) Y^x_s D_tX^x_s
+ \alpha^\prime(X^x_s) D_tY^x_s ds
\]
The solution is
\begin{equation}\label{eq:dermalflot2}
	D_tY^x_T = \dfrac{Y^x_T}{Y^x_t}\displaystyle\int_t^T \alpha^{\prime \prime} (X^x_s)Y^x_s ds.
\end{equation}
As in the previous section, we present the computation under the assumption that
\(\Psi\) is smooth. However, the final result remains true even if \(\Psi\) is
only assumed to be measurable and bounded (see \cite{fournier_etal_1999} for more details).
Using \eqref{eq:expressiongenedelta}, we formally derive with respect to \(x\) and obtain
\[
\dfrac{d^2}{dx^2} \mathbb{E}\Psi(X_T^x)
= \underbrace{\mathbb{E}\left[\dfrac{\Psi^\prime(X^x_T)}{T}Y^x_T\int_0^TY^x_{t}dW_{t}\right]}_{\Gamma_1(x)} + \underbrace{\mathbb{E}\left[\dfrac{\Psi(X^x_T)}{T}\int_0^T Z^x_{t}dW_{t}\right]}_{\Gamma_2(x)}.
\]
The main steps to obtain a tractable unbiased expression of \(\Gamma_2\)
are identical to the ideas used in Section~\ref{subsec:22}. We use the one dimensional setting to remove the stochastic integral and \eqref{eq:dermalflot2} to obtain
\begin{multline}
\Gamma_2(x)  =
\mathbb{E}\left[\dfrac{\Psi(X_T^x)}{T}\left(
W_T Y^x_T \int_0^T \alpha^{\prime \prime}(X^x_t)Y^x_t dt
- \int_0^T W_t \left(Y^x_t \right)^2
\alpha^{\prime \prime}(X^x_t)dt\right.\right.\\
-
\left.\left.\int_0^T
\int_0^t
W_t \alpha^{\prime }(X^x_t)
\alpha^{\prime \prime}(X^x_u) Y^x_u Y^x_t
du
dt\right)\right]
\end{multline}
To simplify \(\Gamma_1\), we apply the Malliavin integration by part formula and \eqref{eq:flotmalliavin}
\[
\begin{split}
\Gamma_1(x)
& = \dfrac{1}{T^2}\mathbb{E}\left[\int_0^T D_t(\Psi(X^x_T))Y^x_t\int_0^TY^x_s dW_s dt\right] \\
& = \dfrac{1}{T^2}\mathbb{E}\left[\Psi(X^x_T)\delta\left(Y^x_t\int_0^TY^x_s dW_s\right)\right].
\end{split}
\]
Finally, we have to make explicit the divergence operator. We apply
\cite[Prop. 1.3.3]{MR2200233} to obtain
\[
\begin{split}
\delta\left(Y^x_t\int_0^TY^x_s dW_s\right) & = \int_0^TY^x_s dW_s \delta\left(Y^x_t\right) -
\int_0^T D_t\left(\int_0^TY^x_s dW_s\right) Y^x_t dt\\
\mbox{ and }D_t\left(\int_0^TY^x_s dW_s\right) & = Y^x_t + \int_t^T D_tY^x_s dW_s.
\end{split}
\]
We again simplify the stochastic integral
\[
\begin{split}
\int_t^T D_tY^x_s dW_s  =& \left(D_tY^x_T\right) W_T
- \int_t^T
W_s
\alpha^{\prime \prime}(X^x_s)
Y^x_s
\dfrac{Y^x_s}{Y^x_t}
ds
\\
&
-
\int_t^T
W_s
\dfrac{\alpha^\prime (X^x_s) Y^x_s }{Y^x_t}
\int_t^s \alpha^{\prime \prime}(X^x_u) Y^x_u du
ds.
\end{split}
\]
Finally, denoting \(U_1\), \(U_2\) and \(U_3\)
three uniform independent random variables, independent of \(W\)
and 
using 
\eqref{eq integunif}, we obtain
\begin{multline*}
\dfrac{d^2}{dx^2} \mathbb{E}\Psi(X_T^x) = \mathbb{E}\left[
\Psi(X^x_T)
\left(
\dfrac{x^2}{T^2}
-
\dfrac{2x}{T^2}   W_T  Y^x_T
+
\dfrac{1}{T^2}
\left(W_T \right) ^2  \left(Y^x_T \right) ^2
\right. \right.
\\
\left. \left.
+
\dfrac{2 x}{T}
W_{U_1 T}
\alpha^{\prime } (X^x_{U_1 T}) Y^x_{U_1 T}
\right. \right.
\\
\left. \left.
-
\dfrac{1}{T}
\left(Y^x_{U_1 T}\right)^2
+
\left(
U_1
-
1
\right)
W_{U_1 T}
\alpha^{\prime \prime}(X^x_{U_1 T}) \left( Y^x_{U_1 T} \right)^2
\right. \right.
\\
\left. \left.
+
W_{U_1 T}
\alpha^{\prime } (X^x_{U_1 T})
W_{U_2 T}
\alpha^{\prime } (X^x_{U_2 T})
Y^x_{U_1 T} Y^x_{U_2 T}
\right. \right.
\\
\left. \left.
-
\dfrac{2 W_T}{T}
W_{U_1 T}
\alpha^{\prime } (X^x_{U_1 T} ) Y^x_T Y^x_{U_1 T}
+
W^x(T)
\left(
1 - U_1
\right)
\alpha^{\prime \prime} (X^x_{U_1 T} ) Y^x_T Y^x_{U_1 T}
\right. \right.
\\
\left. \left.
+
U_1 T
\left(
U_1 U_2
-
1
\right)
W_{U_1 T}
\alpha^\prime (X^x_{U_1 T} )
\alpha^{\prime \prime} (X^x_{U_1 U_2 T})
Y^x_{U_1 U_2 T}
Y^x_{U_1 T}
\right)
\right].
\end{multline*}
Similarly to Sections~\ref{subsec:21} and \ref{subsec:22}, we  
apply Girsanov theorem and \eqref{eq deriveeflot}.
We change in the previous expression
\begin{align*}
X_s^x &\rightarrow \tilde{B}^x_s\\
W_s &\rightarrow \tilde{B}_s^x - s\alpha(\tilde{B}_{U_ks}^x),
\end{align*}
with a uniform random variable \(U_k\), independent of the random objects
previously introduced.
\begin{multline*}
\dfrac{d^2}{dx^2} \mathbb{E}\Psi(X_T^x) = \mathbb{E}\left[
\Psi(\tilde{B}^x_T)
\left(
\dfrac{x^2}{T^2}
-
\dfrac{2x}{T^2}  \left(\tilde{B}^x_T - T\alpha(\tilde{B}^x_{U_3T})\right)  
\exp\left(\int_0^T\alpha^\prime(\tilde{B}^x_\theta)d\theta\right)\right. \right.\\
+
\dfrac{1}{T^2}
\left(\tilde{B}^x_T -  T\alpha(\tilde{B}^x_{U_3T})\right) 
\left(\tilde{B}^x_T -  T\alpha(\tilde{B}^x_{U_4T})\right)
\exp\left(\int_0^T2\alpha^\prime(\tilde{B}^x_\theta)d\theta\right)
\\
+
\dfrac{2 x}{T}
\left(\tilde{B}^x_{U_1 T}-U_1T\alpha(\tilde{B}^x_{U_1U_3T})\right)
\alpha^{\prime } (\tilde{B}^x_{U_1 T}) 
\exp\left(\int_0^{U_1 T}\alpha^\prime(\tilde{B}^x_\theta)d\theta\right)
\\
-
\dfrac{1}{T}
\exp\left(\int_0^{U_1T}2\alpha^\prime(\tilde{B}^x_\theta)d\theta\right)\\
+
\left(
U_1
-
1
\right)
\left(\tilde{B}^x_{U_1 T}-U_1T\alpha(\tilde{B}^x_{U_1U_3T})\right)
\alpha^{\prime \prime}(\tilde{B}^x_{U_1 T}) 
\exp\left(\int_0^{U_1T}2\alpha^\prime(\tilde{B}^x_\theta)d\theta\right)
\\
+
\left(\tilde{B}^x_{U_1 T}-U_1T\alpha(\tilde{B}^x_{U_1U_3T})\right)
\alpha^{\prime } (\tilde{B}^x_{U_1 T})
\left(\tilde{B}^x_{U_2 T} - U_2T\alpha(\tilde{B}^x_{U_2U_4T})\right)
\alpha^{\prime } (\tilde{B}^x_{U_2 T})\\
\times
\exp\left(\int_0^{U_1T}\alpha^\prime(\tilde{B}^x_\theta)d\theta\right)
\exp\left(\int_0^{U_2T}\alpha^\prime(\tilde{B}^x_\theta)d\theta\right)
\\
-
\dfrac{2}{T} \left(\tilde{B}^x_T - T\alpha(\tilde{B}^x_{U_3T})\right)
\left(\tilde{B}^x_{U_1T} - U_1T\alpha(\tilde{B}^x_{U_1U_4T})\right)
\alpha^{\prime } (\tilde{B}^x_{U_1 T} ) \\
\times\exp\left(\int_0^{T}\alpha^\prime(\tilde{B}^x_\theta)d\theta\right)
\exp\left(\int_0^{U_1T}\alpha^\prime(\tilde{B}^x_\theta)d\theta\right)
\\
+
\left(\tilde{B}^x_T - T\alpha(\tilde{B}^x_{U_3T})\right)
\left(
1 - U_1
\right)
\alpha^{\prime \prime} (\tilde{B}^x_{U_1 T} ) \\
\times\exp\left(\int_0^{T}\alpha^\prime(\tilde{B}^x_\theta)d\theta\right)
\exp\left(\int_0^{U_1T}\alpha^\prime(\tilde{B}^x_\theta)d\theta\right)
\\
+
U_1 T
\left(
U_1 U_2
-
1
\right)
\left(\tilde{B}^x_{U_1 T}-U_1T\alpha(\tilde{B}^x_{U_1U_3T})\right)
\alpha^\prime (\tilde{B}^x_{U_1 T} )
\alpha^{\prime \prime} (\tilde{B}^x_{U_1 U_2 T})\\
\left.
\left.
\exp\left(\int_0^{U_1U_2T}\alpha^\prime(\tilde{B}^x_\theta)d\theta\right)
\exp\left(\int_0^{U_2T}\alpha^\prime(\tilde{B}^x_\theta)d\theta\right)
\right)\middle| N\cap D = \emptyset
\right].
\end{multline*}
To conclude, each term on the form \(\exp(-\int_0^s \beta(\tilde{B}^x_\theta)d\theta)\)
is replaced by 
\(
\mathbb{E}\mathbbm{1}_{\{N^j\cap D^j = \emptyset\}}
	\)
	for appropriate Poisson processes \(N^j\) and hypograph \(D^j\)
	(similar terms are expressed in details p.~\pageref{pagepourlexpo}).\label{pagepourlegamma}

\subsection{Simulation of the Poisson Process} \label{subsec:pp_theo_simulation}
We have recalled in Section~\ref{subsec:21} the details of the algorithm developped in \cite{BPR06} to simulate exact paths of the solution
of \eqref{eq:edssigma1}. The main point is the following. Consider a
function \(\varphi\) with values in \([0,K]\), \(\exp(-\int_0^T\varphi(\tilde{B}^x_\theta)d\theta)\)
is the probability that \(N\cap D = \emptyset\), where
\(N\) is a Poisson process with unit intensity on \([0,T]\times[0,K]\)
independent of \(\tilde{B}^x\). The hypograph \(D\) of \(\varphi(\tilde{B}^x_\theta)\) is defined by \eqref{eq:defhypograph}.

For the rejection procedure, we simulate the Poisson process
\((t_1,y_1)\), \(\cdots\), \((t_{n(\omega)},y_{n(\omega)})\) and 
the Brownian bridge at the times \(t_1, \cdots, t_{n(\omega)}\).
If there exists \(j\in [1,n(\omega)]\) such that 
\(y_j < \varphi(\tilde{B}^x_{t_j})\), the Brownian bridge path is rejected.

In \cite{BPR06}, 
the Poisson process
is generated on \([0,T]\times[0,K]\). The result is \(((t_1,y_1),\cdots,(t_{n(\omega)},y_{n(\omega)}))\). Then, the authors simulate the Brownian bridge at time \(t_1\), \(t_2\), \(\cdots\), \(t_{n(\omega)}\) and evaluate if \(N\cap D\) is empty or not.  
In the present paper, we propose two variants of the algorithm. For both variants, 
immediatly after the simulation of one point \((t_j,y_j)\), we simulate \(\tilde{B}^x_{t_j}\). If \(y_j < \varphi(\tilde{B}^x_{t_j})\), we 
have to reject  the Brownian bridge path. So, we
do not need to simulate the full Poisson process \(N\) and stop immediatly the algorithm.
There is two simple variants for the simulation of the Poisson process: first, by increasing times (\(t_1 < t_2 < \cdots < t_{n(\omega)}\)). Second, by increasing ordinates (\(y_1<y_2<\cdots<y_{n(\omega)}\)). 
 This last variant aims to reject as fast as possible the Brownian bridge trajectory. Roughly speaking, 
smaller is the  ordinate, higher is the probability to be below \(\varphi(\tilde{B}^x)\).
We numerically compare the efficiency of the both variants in Section~\ref{sec_num_increasing_ordinate}.

\subsection{A truncated algorithm}\label{subsec:truncatedalgorithm}
The increasing ordinates variant should start, even if we do not know an explicit upper bound \(K\) to \(t\mapsto\varphi(\tilde{B}^x_t)\). 
We propose to extend the Beskos et al. algorithm to SDE with drift \(\alpha\), such that \(\limsup_{y\rightarrow-\infty}\varphi(y) = \limsup_{y\rightarrow+\infty}\varphi(y) = \infty \). According to \cite{BPR06}, \(\varphi=(\alpha^2 + \alpha^\prime)/2\).

For any \(L>0\), we denote \(N^L\) a Poisson process with unit intensity on 
\([0,T]\times[0,L]\).
Our truncated algorithm is 
stopped and we  
accept a path of the Brownian bridge if \(N^{\tilde{K}(\omega)}\cap D = \emptyset\), where  \(\tilde{K}(\omega) \leq \sup_{\theta\in[0,T]}\{\varphi(\tilde{B}^x_\theta)\}\).
Larger is \(\tilde{K}\), smaller is the probability to wrongly accept a path, but slower is the algorithm. A reasonnable choice of \(\tilde{K}(\omega)\) is
\[
\tilde{K}(\omega) \geq \max\{K,\varphi(\tilde{B}^x_T),\varphi(\inf_{0\leq s\leq T}\tilde{B}^x_s)\},
\]
where \(K\) is an a priori threshold.
Our algorithm is no more unbiased. However, Proposition~\ref{prop:controlerreur} gives an upper bound of the error 
in the approximation of \eqref{eq:prix}.

\subsection{Theoretical control of the error} \label{subsec:control_error}
\begin{prop}
	\label{prop:controlerreur}
	Let \(X_T^x\) solution of \eqref{eq:edssigma1} and \(X_T^{x,K}\) its 
	approximation obtained by the truncated rejection procedure 
	presented in Section~\ref{subsec:truncatedalgorithm}. Precisely, the Brownian bridge path is accepted if there is no point of an independent Poisson process on \([0,T]\times[0,K]\) in the hypograph \(D\) of \(\varphi\) (given by \eqref{eq:defhypograph}). Then:
\begin{enumerate}[a.]
	\item 
	\begin{multline}
	\label{eq:controleerreur}
	\left|\mathbb{E}\Psi\left(X_T^x\right) - \mathbb{E}\Psi\left(X^{x,K}_T\right)\right|\\
	\leq 
	\sqrt{\mathbb{E} \left[\Psi^2 \left(  \tilde{B}^x_T \right)\right]}
	\left(\dfrac{\mathbb{P}\left(\sup_{0\leq\theta\leq T}\varphi(\tilde{B}^x_\theta)>K\right)}{p_K\sqrt{p_\infty}} + \dfrac{\sqrt{\mathbb{P}\left(\sup_{0\leq\theta\leq T}\varphi(\tilde{B}^x_\theta)>K\right)}}{p_K} \right),
	\end{multline}
	where \(p_K\)  denotes the probability to accept a Brownian bridge path with the truncated algorithm at level \(K\),
	\begin{equation}\label{eq:probaaccepttruncK}
	p_K = \mathbb{E}\left[\exp\left(-\int_0^TK\wedge\varphi(\tilde{B}^x_\theta)d\theta\right)\right]
	\end{equation}
	and \(p_\infty\) is given by
	\begin{equation}\label{eq:probaacceptfull}
	p_\infty = \mathbb{E}\left[\exp\left(-\int_0^T\varphi(\tilde{B}^x_\theta)d\theta\right)\right].
	\end{equation}
	\item If moreover \(\Psi\) is bounded, 
		\begin{equation}
		\label{eq:controleerreurpsibornee}
		\left|\mathbb{E}\Psi\left(X_T^x\right) - \mathbb{E}\Psi\left(X^{x,K}_T\right)\right|
		\leq 
\dfrac{2\left\|\Psi\right\|_\infty}{p_K}\mathbb{P}\left(\sup_{0\leq\theta\leq T}\varphi(\tilde{B}^x_\theta)>K\right).
		\end{equation}
	\end{enumerate}

\end{prop}
\begin{rem}
	\begin{enumerate}
	\item If \(\limsup_{y\rightarrow-\infty} \varphi(y) = \limsup_{y\rightarrow+\infty} \varphi(y) = +\infty\),
	for any Brownian bridge, the probability 
	to wrongly accept the trajectory is
	positive. 
	However, Proposition~\ref{prop:controlerreur} gives a control
	of the error.
	\item The result of Proposition~\ref{prop:controlerreur} still holds true
	if we use the variant of the algorithm with the simulation of the minimum of the Brownian bridge (see point \ref{extensionavecmin}) of Remark~\ref{rem:relachehypsurphi} and \cite{BPR06} ).
	Numerical results for this variant are given in Section~\ref{sec_num_increasing_ordinate}.
	\item If we have a control of the asymptotic behavior of \(\varphi\)
	(e.g. a polynomial growth at infinity), we deduce that the error 
	of truncation decreases exponentially fast to \(0\) with \(K\).
	\end{enumerate}
\end{rem}

\begin{proof}
	We denote by \(N^K\) a Poisson process on \([0,T]\times[0,K]\)
	and by \(N\) a Poisson process on \([0,T]\times\mathbb{R}_+\).
	Thanks to \eqref{eq:egalloiavecpoissonhypo}, we have
\begin{align*}
	\mathbb{E}\Psi\left(X_T^x\right) &= \mathbb{E}\left[\Psi(\tilde{B}^x_T)\middle| N\cap D(\omega) = \emptyset\right]\\
	& = \dfrac{\mathbb{E}\left[\Psi(\tilde{B}^x_T)\exp\left(-\int_0^T\varphi(\tilde{B}^x_\theta)d\theta\right)\right]}{\mathbb{E}\left[\exp\left(-\int_0^T\varphi(\tilde{B}^x_\theta)d\theta\right)\right]}\\
\mathbb{E}\Psi\left(X^{x,K}_T\right) &= \mathbb{E}\left[\Psi(\tilde{B}^x_T)\middle| N^K\cap D(\omega) = \emptyset\right],\\
& = \dfrac{\mathbb{E}\left[\Psi(\tilde{B}^x_T)\exp\left(-\int_0^TK\wedge\varphi(\tilde{B}^x_\theta)d\theta\right)\right]}{\mathbb{E}\left[\exp\left(-\int_0^TK\wedge\varphi(\tilde{B}^x_\theta)d\theta\right)\right]}.
\end{align*}
We denote by \(p_K\) and \(p_\infty\) the probabilities to
accept a Brownian bridge path with the truncated algorithm at level \(K\)
and with the exact algorithm (see \eqref{eq:probaaccepttruncK} and \eqref{eq:probaacceptfull}).
%\begin{align}
%p_K & = \mathbb{E}\left[\exp\left(-\int_0^TK\wedge\varphi(\tilde{B}^x_\theta)d\theta\right)\right],\\
%p_\infty & = \mathbb{E}\left[\exp\left(-\int_0^T\varphi(\tilde{B}^x_\theta)d\theta\right)\right].
%\end{align}
Thus, a control of the error is
\begin{align*}
\mathrm{err}_K & = \left|\mathbb{E}\Psi\left(X_T^x\right) - \mathbb{E}\Psi\left(X^{x,K}_T\right)\right|\\
& \leq 
\left|
\dfrac{1}{p_{\infty}}
-
\dfrac{1}{p_K}
\right|
\mathbb{E} \left[ \left|\Psi \left( \tilde{B}^x_T\right)\right| \exp\left(-\int_0^T\varphi(\tilde{B}^x_\theta)d\theta\right)  \right]\\
&+
\dfrac{1}{p_K}
\mathbb{E} \left[ \left|\Psi \left(  \tilde{B}^x_T \right)\right|
\left( 
\exp\left(-\int_0^TK\wedge\varphi(\tilde{B}^x_\theta)d\theta\right) -\exp\left(-\int_0^T\varphi(\tilde{B}^x_\theta)d\theta\right) 
\right)
\right].
\end{align*}
We apply Cauchy-Schwarz inequality and use that \(x^2\leq x\) for \(0\leq x\leq 1\)
\begin{align*}
\mathrm{err}_K  \leq &
\dfrac{p_K-p_\infty}{p_Kp_\infty}
\sqrt{\mathbb{E} \left[ \Psi^2 \left( \tilde{B}^x_T\right)\right]}
\sqrt{p_\infty}\\
&+
\dfrac{1}{p_K}
\sqrt{\mathbb{E} \left[\Psi^2 \left(  \tilde{B}^x_T \right)\right]}
\sqrt{
	\mathbb{E}
	\left[
	\exp\left(-\int_0^TK\wedge\varphi(\tilde{B}^x_\theta)d\theta\right) 
- 
\exp\left(-\int_0^T\varphi(\tilde{B}^x_\theta)d\theta\right) 
\right]}\\
\leq & \sqrt{\mathbb{E} \left[\Psi^2 \left(  \tilde{B}^x_T \right)\right]}
\left(\dfrac{p_K-p_\infty}{p_K\sqrt{p_\infty}} + \dfrac{\sqrt{p_K-p_\infty}}{p_K} \right).
\end{align*}
We finally observe that 
\[
p_K-p_\infty\leq \mathbb{P}\left(\sup_{0\leq\theta\leq T}\varphi(\tilde{B}^x_\theta)>K\right).
\]
The proof under the assumption that \(\Psi\) is bounded is
very similar and simpler.
It is left to the reader.
\end{proof}

%\ladim{
\subsection{Extension in finite dimension}

One can easily extend the algorithm to a multi-dimensional setting
under restrictive assumptions. For instance, assume that 
the drift  derives from a potential, that is, there exists a function \(\xi:\mathbb{R}^d\mapsto\mathbb{R}^d\) such that
\[
X_T = X_0 + \int_0^T \nabla \xi(X_s)ds + W_T.
\]
%where
%\(X,W\in\mathbb{R}^d\).
Girsanov Theorem gives the Radon-Nikodym derivative:
\[
\exp\left(\int_0^T\sum_{j=1}^d\frac{\partial}{\partial x_j} \xi(W_s)dW^j_s - \int_0^T \sum_{j=1}^d\left(\frac{\partial}{\partial x_j}\xi(W_s)\right)^2ds\right).
\]
Similarly to the one-dimensional case, 
our assumption allows one to introduce the antiderivative of \(\nabla\xi\) 
in order to remove the stochastic integral in the previous expression:
\[
\xi(W_T) = \xi(W_0) + \int_0^T\sum_{j=1}^d\frac{\partial}{\partial x_j} \xi(W_s)dW^j_s + \dfrac{1}{2}\int_0^T\sum_{j=1}^d\frac{\partial^2}{\partial x_j^2} \xi(W_s)ds.
\]
%We can simulate the final value in the same way by finding an upper bound (with a constant multiplicative term) of \(\exp\left( \Psi(W_T) - \Psi(W_0)\right)\) by a known random variable.
So, as in the one-dimensional case, we simulate a Brownian bridge with 
final distribution
\[
\mathbb{P}(\tilde{B}^x_T\in d\theta)  = C\exp\left(- \dfrac{\sum_j(\theta_j-x_j)^2}{2T}+\xi(\theta) \right)d\theta.
\]
We only have to replace the function \(\varphi\) in \eqref{eq:defdephi} by \(\frac{1}{2}  \sum_{j=1}^d (\frac{\partial}{\partial x_j}\xi(y))^2 + \frac{\partial^2}{\partial x_j^2} \xi(y) \).

The evaluation of the derivatives is more difficult to extend.
Equation \eqref{eq:expressiongenedelta} has an equivalent in any
finite dimension \cite{fournier_etal_1999}. We can also write
\eqref{eq:intydw} but, for instance, it is not 
easy to write the transform of each component of the vector \(Y_T^xW_T\) as the exponential of an integral.

\section{The detailed algorithm for the CIR Model}\label{sec:cirmodel}

	This section is devoted to the extension of our algorithm to the simulation of the Cox Ingersoll Ross (CIR) process, a popular model in finance (for short rates or volatility for stochastic volatility model on asset, etc.)     This process satisfies
\begin{equation}\label{eq:CIR}
V_T = V_0 + \int_0^T \kappa\left(V_\infty - V_t\right) dt + \varepsilon\int_0^T\sqrt{V_t}dW_t
\end{equation}
where \(\kappa\), \(V_\infty\) and \(\varepsilon\) are fixed constants.
Usually, the parameter \(d = \frac{4\kappa V_\infty}{\varepsilon^2}\) 
is  called the degree of the CIR process.
It is known  that 
\(\mathbb{P}(\inf_{\theta\in[0,T]}V_\theta > 0) = 1\)
iff \(d \geq 2\) (see e.g. \cite{livreAlfonsi}).
We assume it is fulfilled.
We apply the Lamperti transform to the process \(V\), 
that is we set
\[
X_t = \frac{2\sqrt{V_t}}{\varepsilon} =: \eta(V_t).
\]
The process \(X\) satisfies the SDE
\begin{align}
\nonumber
dX_t & = \eta^\prime(V_t)dV_t + \frac{1}{2}\eta^{\prime\prime}(V_t)d\left\langle V\right\rangle_t\\
\nonumber
& =\frac{1}{\varepsilon \sqrt{V_t}} \left(\kappa (V_{\infty} - V_t ) dt + \varepsilon \sqrt{V_t} dW_t \right) - \frac{\varepsilon^{2} V_t}{4\varepsilon V_t^{3/2}}dt\\
& = \left(\frac{1}{X_t} \left( \frac{2 \kappa V_{\infty}}{\varepsilon ^{2}}
- \frac {1}{2}\right) - \frac{\kappa X_t}{2} \right) dt + dW_t.
\label{eq:defxcir}
\end{align}
It is an SDE of type \eqref{eq:edssigma1} 
with
\[
\alpha(y) =  \frac{1}{y} \left( \frac{2 \kappa V_{\infty}}{\varepsilon ^{2}} - \frac {1}{2}\right) - \frac{\kappa y}{2} \mbox{ for } y > 0.
\]
The associated function \(\varphi\) defined by \eqref{eq:defdephi} is
\begin{equation*}
  \varphi(y) = \left(\left(\frac{2\kappa V_\infty}{\varepsilon^2} - 1\right)^2 - \frac{1}{4} \right)
\frac{1}{2y^2} + \frac{\kappa^2}{8} y^2 - \frac{\kappa^2V_\infty}{\varepsilon^2} \mbox{ for } y > 0.
\end{equation*}
The function \(\varphi\) is bounded below on $(0,+\infty)$ iff
\[
\left(\frac{2\kappa V_\infty}{\varepsilon^2} - 1\right)^2 \geq \frac{1}{4}
\]
or equivalently that the degree \(d\)  of the CIR
satisfies
\(
d\in (0,1]\cup [3,\infty).
\)
In this paper, we assume \(d\geq 3\).
\begin{rem}
	In Section~\ref{sec:section2}, the drift \(\alpha\) is defined on
	\(\mathbb{R}\). However, a classical Feller test proves that 
	the process \(X^x\), solution of \eqref{eq:defxcir} starting from \(x>0\), never hits \(0\) almost surely. Formally, if we put \(\alpha(y) = \varphi(y) = +\infty\) for all \(y\leq 0\), the Brownian bridge paths
	\(\tilde{B}\) taking values in \(\mathbb{R}_-\) are almost surely rejected.
\end{rem}
\subsection{Final Value \label{final_value_sim_cir}}
In the first step, we generate 
the final value \(\tilde{B}^x_T\) according to \eqref{eq:distribfinBBrid}. Its
density is
\[
h(y)  = R y^c \exp\left(-\dfrac{\left(y-\hat{x}\right)^2}{2\sigma^2}\right)\mathbbm{1}_{y\geq 0}
\]
with:
\begin{align*}
c  & = \dfrac{2\kappa V_\infty}{\epsilon^2} - \dfrac{1}{2}, &
\hat{x}  & =2\sigma^2 \frac{x}{2T}, &
x  & = \dfrac{2}{\epsilon}\sqrt{V_0},
& \sigma^2 & = \frac{1}{\dfrac{\kappa}{2}+\dfrac{1}{T}}
\end{align*}
and \(R\) is a normalisation.
Setting \(
\bar{x}  =  \frac{ \hat{x} + \sqrt{\hat{x}^2 + 4 c \sigma^2} }{2}
\), there exists \(C>0\) such that
\[
\forall y, \quad h(y)\leq C \exp\left(-\dfrac{\left(y - \bar{x}\right)^2}{2\sigma^2}\right),
\]
and we use the classical rejection procedure for random variables.
\subsection{Simulation of the minimum}
The second step consists in generating the random variables \((m,t_m)\), where
\[
m  = \inf_{0\leq t\leq T}\left\{\tilde{B}^x_t \middle| \tilde{B}_0 = x ,\tilde{B}^x_T = Y \right\} \quad\mbox{ and }\quad 
 \tilde{B}^x_{t_m}  = m.
\]
This law is known (see for instance Karatzas-Shreve~\cite[p. 102]{KS91})
\[
\mathbb{P}\left[m\in d\alpha , {t_m} \in ds \middle| \tilde{B}^x_T = Y\right] =
\dfrac{\alpha(\alpha-Y)}{\sqrt{s^3(T - s)^3}}\exp\left(-\dfrac{\alpha^2}{2s}-\dfrac{(\alpha - Y)^2}{2(T - s)}\right)d\alpha ds.
\]
In Beskos et al. \cite[Prop. 2]{BPR06}, the detailled random variables used to
simulate \((m,{t_m})\) are given: the authors only need to simulate  uniform, exponential and Inverse Gaussian distributions
(see Devroye~\cite[p.149]{devroye} for an efficient way to simulate Inverse Gaussian distributions).
\subsection{Simulation of the Poisson process }
We apply the method detailed in Section~\ref{subsec:pp_theo_simulation}.
We generate
\(z_1\sim\mathcal{E}(T)\), \(t_1\sim\mathcal{U}(0,T)\), \(\tilde{B}^x_{t_1}\) conditioned by
\(\tilde{B}^x_0 ,\tilde{B}^x_T , m, {t_m}\). If \(\varphi(\tilde{B}^x_{t_1}) > z_1\), we reject the
trajectory. Else, we generate \(z_2 - z_1 \sim \mathcal{E}(T)\), \(t_2 \sim \mathcal{U}(0,T)\),
\(\tilde{B}^x_{t_2}\) conditioned by \(\tilde{B}^x_0, \tilde{B}^x_{t_1}, \tilde{B}^x_T , m, {t_m}\). If
\(\varphi(\tilde{B}^x_{t_2}) > z_2\), we reject the trajectory, etc.

\subsection{Stopping condition \label{cir_stopping_condition}}
In this example, \(\sup_{y\geq m(\omega)}\varphi(y) = +\infty\). 
So, we use the truncated algorithm presented in Section~\ref{subsec:truncatedalgorithm}.
We simulate the Poisson process on \([0,T]\times[0,\tilde{K}(\omega)]\)
with 
\begin{equation}\label{eq:lienktildek}
\tilde{K}(\omega) \geq \max\{K,\varphi(\tilde{B}^x_T),\varphi(m)\},
\end{equation}
where \(K\) is a fixed a priori threshold.

\section{Numerical Results}\label{sec:numerical}
In this Section, we present the numerical results.
We first apply the algorithm to 
an academic example related to Orstein-Uhlenbeck process (Section~\ref{subsec:MOU}). 
The drift \(\alpha\) is constructed such that its associated 
function \(\varphi\)
satisfies 
\(\limsup_{y\rightarrow\infty}\varphi(y)<\infty\).
In Section~\ref{subsec:surem}, the drift \(\alpha\) is
constructed in such a way that
the associated function \(\varphi\) satisfies  \(\limsup_{y\rightarrow-\infty}\varphi(y) = \limsup_{y\rightarrow\infty}\varphi(y)=\infty\).
Finally, Section~\ref{subsec:cir_an} is devoted to the 
CIR process (see Section~\ref{sec:cirmodel}), i.e. an example with a non Lipschitz continuous drift \(\alpha\).

We use the algorithms to approximate  quantities \eqref{eq:prix}, \eqref{eq:delta} and \eqref{eq:gamma} for smooth and nonsmooth functions \(\Psi\). We compare the efficiency of our algorithm to the use 
of a classical Euler scheme and finite difference approximation
of the derivatives.
\subsection{An academic example: a modified Ornstein Uhlenbeck}\label{subsec:MOU}
\subsubsection{Definition} \label{sec:turem_drift}
We introduce the process \((X^x_t,t\geq 0)\), solution of
\begin{equation}\label{eq:defturem}
dX_t^x  = \left(-M \left(X_t^x + \dfrac{1}{2}\right) \mathbbm{1}_{X_t^x \leq -1} +\frac{M}{2} (X_t^x)^2 \mathbbm{1}_{-1 \leq X_t^x \leq 0}\right)dt + dW_t,
\end{equation}
where \(M\geq 1/2\) is a fixed parameter.
The process \(X^x\) is solution of an SDE of type \eqref{eq:edssigma1} with a  drift \(\alpha\in C^1(\mathbb{R})\).
Its associated function \(\varphi\) is
\[
\varphi(y) = \begin{cases}
0 \quad &\mbox{ if } \quad y\geq 0\\
\dfrac{M^2y^4}{8} + \dfrac{My}{2} \quad &\mbox{ if }\quad -1 \leq y\leq 0\\
\dfrac{M^2}{2}\left(y +\dfrac{1}{2}\right)^2  - \dfrac{M}{2} \quad &\mbox{ if }\quad  y\leq -1.
\end{cases}
\]
It satisfies \[
\lim_{y \rightarrow-\infty}\varphi(y) = +\infty \quad \quad and 
\quad\quad
\lim_{y \rightarrow+\infty}\varphi(y) < \infty.
\]
Then, SDE \eqref{eq:defturem} satisfies the assumptions made in Section~\ref{sec:section2} and 
we are in position to apply our unbiased algorithm to approximate 
 \(\mathbb{E}(\Psi(X_T^x))\), \(\frac{d}{dx}\mathbb{E}(\Psi(X_T^x))\) and 
 \(\frac{d^2}{dx^2}\mathbb{E}(\Psi(X_T^x))\) for general functions \(\Psi\).

\subsubsection{Algorithmic optimisation of computation time} \label{sec_num_increasing_ordinate}
We have discussed in Section~\ref{subsec:pp_theo_simulation} two variants to simulate the Poisson process N used to reject (or accept) the Brownian bridge paths.
\begin{itemize}
	\item \textbf{variant 1} by increasing times: a realisation of \(N\), say   \(\{(t_1,y_1), \cdots,\)
	\( (t_{n(\omega)},y_{n(\omega)})\}\), satisfies
	\(t_1<t_2<\cdots<t_{n(\omega)}\).
		\item \textbf{variant 2} by increasing ordinates: \(\{(t_1,y_1), \cdots, (t_{n(\omega)},y_{n(\omega)})\}\) satisfies
		\(y_1<y_2<\cdots<y_{n(\omega)}\).
\end{itemize}
In this part, we compare the efficiency of the two variants.
They only differ by the computation time used to accept a Brownian bridge path.
Figure~\ref{fig:ornstein} represents the time of simulation as a function of the final time \(T\). The size of the sample is \(N_{\textrm{MC}} = 1e6\) and the parameters are \(x=0.04\), \(M=0.5\).

We observe that the  times of simulation are very close for 
small values of \(T\); they both increase exponentially and,  clearly, the rate is smaller for variant 2 than variant 1.
\begin{figure}[ht]
	\begin{center}
		\includegraphics[height=13cm,angle=-90]{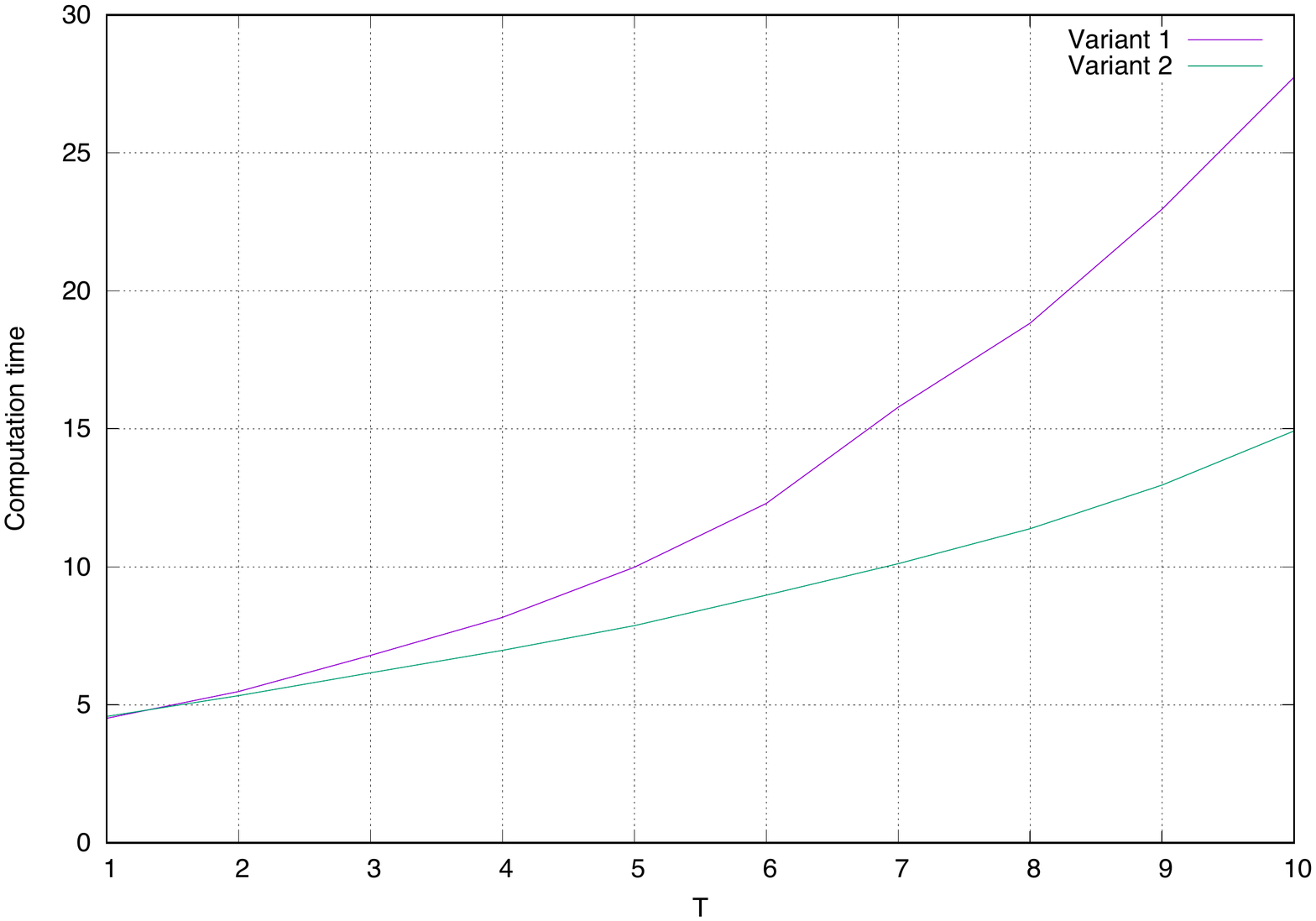}
		\caption{Comparison of the  times of simulation for two methods to generate the
			Poisson process: variant 1 (increasing times) and variant 2 (increasing ordinates). Times of simulation (in seconds) are given in function of the final time \(T\).
			The process \(X^x\) solves \eqref{eq:defturem}. The parameters are \(M=0.5\),  \(x=0.04\) and \(N_{\textrm{MC}} = 1e6\).}
		\label{fig:ornstein}
	\end{center}
\end{figure}

We then fix the final time \(T=1\) and  change the parameter \(M\)
in the drift \(\alpha\) (see \eqref{eq:defturem}). The times of simulation
of a sample of size \(N_{\textrm{MC}} = 1e6\) are given in Table.~\ref{tab:ratio_tu_time_computation}.
Again, the variant 2 is faster than
variant 1.

\begin{table}[ht!]
\begin{center}
\begin{tabular}{|c|c|c|c|}
 \hline
 \(M\) & time (var. 1)  & time (var. 2) & Ratio
 \tabularnewline
\hline
 1 & 7.79 & 7.74 & 1.01 %\eot{souci. Verifier les donnees}
% Victor a vérifié
 \tabularnewline
\hline
 10 & 31 & 14 & 2.21
 \tabularnewline
\hline
 100 & 254591 & 5148 & 49.4
 \tabularnewline
\hline
\end{tabular}
 \caption{Comparison of the times (in sec.) of simulation for variant 1 (increasing times) and variant 2 (increasing ordinates). We simulate \(N_{\textrm{MC}} = 1e6\) values of \(X_T\) (\(x=0\), \(T=1\), \(M = 1, 10, 100\)).}
\label{tab:ratio_tu_time_computation}
\end{center}
\end{table}

\subsubsection{A comparison of approximations of sensitivities} \label{sec:tuefd_vs_turem}
The unbiased evaluation of the sensitivities \(\frac{d}{dx} \mathbb{E}\Psi(X_T^x)\) and \(\frac{d^2}{dx^2} \mathbb{E}\Psi(X_T^x)\) are the main new results of the paper. 
They are themselves interesting theoretical results. 
However, we aim to compare their efficiency to classical numerical methods.

\paragraph*{Our unbiased estimator }
%It is fully presented in Section~\ref{sec:section2}.
We apply a classical Monte Carlo procedure to evaluate the expressions \eqref{eq:egalloiavecpoissonhypo}, \eqref{eq:deltaavecpoisson} and
the expression p.\pageref{pagepourlegamma} for the second derivative.
We denote the Monte Carlo estimators by
 \[
 \tilde{P}_\Psi(N_{\textrm{MC}}),\quad\quad  \tilde{\Delta}_\Psi(N_{\textrm{MC}}), \quad \quad
 %\mbox{ and } \quad \quad 
 \tilde{\Gamma}_\Psi(N_{\textrm{MC}}).
 \]
There is a unique source of error: the statistical error. It is 
only related to the variance of the expressions we evaluate.
In Table~\ref{tab:turem_2e10}, we present the results for three functions \(\Psi\), two are smooth and the last one is discontinuous.
We put in brackets
 the estimated statistical standard deviation 
 with a sample of size \(N_{\textrm{MC}}=2e10\).

\paragraph*{Standard estimator using Euler scheme and finite difference approximation} We simulate \(X_T^{x,\delta,1}, \cdots, X_T^{x,\delta,N_{\textrm{MC}}}\),  \(N_{\textrm{MC}}\) independent realisations of the explicit Euler scheme (with time step \(\delta\)) to approximate the solution \(X_T^x\) of \eqref{eq:edssigma1}.
The derivatives are approximated with a finite difference scheme. 
That is,
we simulate \(X_T^{x-dx,\delta,1}, \cdots, X_T^{x-dx,\delta,N_{\textrm{MC}}}\)
and \(X_T^{x+dx,\delta,1}, \cdots, X_T^{x+dx,\delta,N_{\textrm{MC}}}\) and use the estimators
\begin{align*}
  \hat{P}_\Psi(N_{\textrm{MC}},\delta) :=& \dfrac{1}{N_{\textrm{MC}}}\sum_{k=1}^{N_{\textrm{MC}}} \Psi(X_T^{x,\delta,k}) \\
  \approx&\mathbb{E}(\Psi(X_T^x)),\\
   \hat{\Delta}_\Psi(N_{\textrm{MC}},\delta,dx) := &
\dfrac{1}{2dxN_{\textrm{MC}}}\left( \sum_{k=1}^{N_{\textrm{MC}}} \Psi(X_T^{x+dx,\delta,k}) - 
 \sum_{k=1}^{N_{\textrm{MC}}} \Psi(X_T^{x-dx,\delta,k})
  \right)  \\
\approx&  \frac{d}{dx}\mathbb{E}(\Psi(X_T^x)),
  \\
  \hat{\Gamma}_\Psi(N_{\textrm{MC}},\delta,dx) := &
  \dfrac{1}{(dx)^2N_{\textrm{MC}}}\left( \sum_{k=1}^{N_{\textrm{MC}}} \Psi(X_T^{x+dx,\delta,k}) - 
  2\sum_{k=1}^{N_{\textrm{MC}}} \Psi(X_T^{x,\delta,k})\right.\\
& \quad \quad\quad \quad \left. +\sum_{k=1}^{N_{\textrm{MC}}} \Psi(X_T^{x-dx,\delta,k})
  \right) \\
\approx&    \frac{d^2}{dx^2}\mathbb{E}(\Psi(X_T^x))
\end{align*}
These approximations are also very simple to simulate and evaluate. 
We now have two sources of error:
\begin{itemize}
	\item biases due to the parameters \(\delta\) and \(dx\).
	\item the statistical error, related to the variance of the quantities we estimate with a Monte Carlo procedure.
\end{itemize}
In practice, we have to carefully choose \(N\), \(\delta\) and \(dx\). 
The best choice is obtained if the bias is close to the statistical error.
It is not easy to reach such a balance: we do not know the bias.

We have chosen two set of parameters,  \(N_{\textrm{MC}} = 1e9\), \(\delta = 0.1\)
and \(dx=0.4\) in Table~\ref{tab:tuefd_step_1em1_stepdf_4em1} and \(N_{\textrm{MC}} = 5e7\), \(\delta = 0.005\) and \(dx=0.1\) in 
Table~\ref{tab:tuefd_step_5em3_stepdf_2em1}.

\begin{table}[ht!]
	\begin{center}
		\begin{tabular}{|c|c|c|c|}
			\hline
			&&&\tabularnewline
			\( \Psi(y)\) & \(\tilde{P}_\Psi(N_{\textrm{MC}})\)
			& \(\tilde{\Delta}_\Psi(N_{\textrm{MC}})\)
		& 
			 \(\tilde{\Gamma}_\Psi(N_{\textrm{MC}})\)
			\tabularnewline&&&\tabularnewline
			\hline
			\( y^2 \)  &  0.900933  (9.0e-6)  &  0.301072 (2.5e-5)   &  1.57485 (5.6e-5)    \tabularnewline
			\(  \exp(-y) \)   &  1.40071 (1.1e-5)  &  -1.16071 (2.8e-5)  &  0.703935 (7.2e-5)   \tabularnewline
			\(  \mathbbm{1}_{y > x} \)   &  0.492925 (3.5e-6)  &  -0.3854 (4.7e-6)   &  -0.0219749 (8.3e-6)  \tabularnewline
			\hline
		\end{tabular}
		\caption{Approximation for \(X^x\) solution of \eqref{eq:defturem} obtained with our unbiased algorithms. 
			\(N_{\textrm{MC}} = 2e10\), \(M = 0.5\), \(x = 0.04\).
			The program runs \(9e4\) seconds.
		}
		\label{tab:turem_2e10}
	\end{center}
\end{table}

\begin{table}[ht!]
\begin{center}
\begin{tabular}{|c|c|c|c|}
\hline
&&&\tabularnewline
\( \Psi(y)\) & \(\hat{P}_\Psi(N_{\textrm{MC}},\delta)\)
& 
\( \hat{\Delta}_\Psi(N_{\textrm{MC}},\delta,dx)\)
& 
\( \hat{\Gamma}_\Psi(N_{\textrm{MC}},\delta,dx)\)
\tabularnewline
 & \(\quad- P_\Psi\)
& 
\( \quad -\Delta_\Psi\)
& 
\( \quad -\Gamma_\Psi\)
\tabularnewline
&&&\tabularnewline
\hline
\(  y^2 \)  &  8.8e-3  (4.1e-5)  &  5.0e-3 (1.1e-4)   &  1.1e-3 (1.1e-3)    \tabularnewline
\(  \exp(-y) \)   &  1.5e-2 (4.9e-5)  &  -2e-2 (1.3e-4)  &  7.0e-3 (1.2e-3)   \tabularnewline
\(  \mathbbm{1}_{y > x } \)   &  -7.1e-5 (1.6e-5)  &  1.1e-2 (3.8e-5)   &  -2.6e-3 (3.9e-4)  \tabularnewline
\hline
\end{tabular}
\caption{
	Error with an Euler scheme with step \(\delta = 0.1\) and
	finite difference approximation with step \(dx=0.4\). 
	\(X^x\) is solution of \eqref{eq:defturem},
	\(N_{\textrm{MC}} = 1e9\), \(M = 0.5\), \(x = 0.04\).
	The program runs \(5.6e3\) seconds.
	}
\label{tab:tuefd_step_1em1_stepdf_4em1}
\end{center}
\end{table}

\begin{table}[ht!]
\begin{center}
\begin{tabular}{|c|c|c|c|}
	\hline
	&&&\tabularnewline
	\( \Psi(y)\) & \(
	\hat{P}_\Psi(N_{\textrm{MC}},\delta)\)
	& 
	\( \hat{\Delta}_\Psi(N_{\textrm{MC}},\delta,dx)\)
	& 
	\( \hat{\Gamma}_\Psi(N_{\textrm{MC}},\delta,dx)\)
\tabularnewline
 & \(\quad- P_\Psi\)
 & 
 \( \quad -\Delta_\Psi\)
 & 
 \( \quad -\Gamma_\Psi\)
 \tabularnewline
 &&&\tabularnewline

	\hline
	\(  y^2 \)   &  4.5e-4  (1.8e-4)  &  3.6e-3 (9.3e-4)   &  -2.2e-3 (1.8e-2)    \tabularnewline
\( \exp(-y) \)   &  8.1e-5 (2.1e-4)  &  -1.1e-3 (1.1e-3)  &  -1.2e-3 (2.1e-2)   \tabularnewline
\( \mathbbm{1}_{y > x} \)   &  6.8e-5 (7.1e-5)  &  2.7e-3 (3.5e-4)   &  -5.0e-4 (7.0e-3)  \tabularnewline
\hline
\end{tabular}
\caption{
		Error with an Euler scheme with step \(\delta = 0.005\) and
		finite difference approximation with step \(dx=0.2\). 
		\(X^x\) is solution of \eqref{eq:defturem},
		\(N_{\textrm{MC}} = 5e7\), \(M = 0.5\), \(x = 0.04\).
		The program runs \(2.9e3\) seconds.
		}
\label{tab:tuefd_step_5em3_stepdf_2em1}
\end{center}
\end{table}

\paragraph{Conclusion}
To obtain an error of the same magnitude with our unbiased estimator,
we have to use between \(N_{\textrm{MC}}=1e5\) and \(N_{\textrm{MC}}=1e6\)  for the rough case (Table~\ref{tab:tuefd_step_1em1_stepdf_4em1})
and between \(N_{\textrm{MC}}=1e6\) and \(N_{\textrm{MC}}=1e7\) for the more precise case (Table~\ref{tab:tuefd_step_5em3_stepdf_2em1}).
The size of the sample obviously depends on the function \(\Psi\) and
the order of the derivative we approximate.
Our algorithm is well adapted for the approximation of \(\Delta_\Psi\)
and \(\Gamma_\Psi\).

In any cases, our algorithm is faster (10 to 100 times faster than the Euler
scheme).

\subsection{Symmetric modified Orstein-Uhlenbeck, convergence of the error of truncation} \label{subsec:surem}
	We test our unbiased algorithm to a second \textit{toy} model.
	We only evaluate in this Section the error due to the truncation
	of the Poisson process. That is, we illustrate the results of
	Section~\ref{subsec:control_error}.
	The comparison with an Euler scheme and finite difference approximation
	of the derivatives are very similar (in terms of complexity and of efficiency) to those obtained in the previous section. 
	Thus, we do not include them for this example.

\subsubsection{Introduction} \label{sec:surem_drift}
We slightly modify the drift introduced in the previous example.
In this part, we put
\begin{multline}\label{eq:surem}
X_T^x = x + \int_{0}^{T}\alpha(X_t^x)dt + W_T\\
\alpha(x)  = -M \left(x + \dfrac{1}{2}\right) \mathbbm{1}_{x \leq -1} +\frac{M}{2} x^2 \mathbbm{1}_{-1 \leq x \leq 1}
+ M \left(x - \dfrac{1}{2}\right) \mathbbm{1}_{x \geq 1}.
\end{multline}
\begin{rem}
	For \(y\leq 0\), the drift \(\alpha(y)\) is identical to the
	drift in the previous example, but instead of putting \(\alpha(y)= 0\) for \(y\geq 0\), the drift is now symmetric.
	The associated function \(\varphi\) satisfies
	\(\lim_{-\infty}\varphi = \lim_{+\infty}\varphi=+\infty\).	
\end{rem}
For any threshold \(K\), we simulate the final
value \(\tilde{B}_T^{x}\), the minimum \(m\) of the Brownian 
bridge on \([0,T]\) and compute \(\tilde{K}(\omega)\) according to \eqref{eq:lienktildek}. We then simulate a Poisson process \(N^K\)
on \([0,T]\times[0,\tilde{K}(\omega)]\) and accept the path if
\(N^K\cap D(\omega)=\emptyset\), where \(D(\omega)\) denotes
the hypograph of \(\varphi(\tilde{B}_t^x)\). We denote by \(\tilde{X}^{x,K}_T\)
the accepted values. We denote by \(p^K\) the probability to accept a Brownian bridge path (see \eqref{eq:probaaccepttruncK}).

We use the notation 
\(\tilde{P}_\Psi(N_{\textrm{MC}},K)\), \(\tilde{\Delta}_\Psi(N_{\textrm{MC}},K)\) and \(\tilde{\Gamma}_\Psi(N_{\textrm{MC}},K)\)
for our Monte Carlo approximations of \eqref{eq:prix}, \eqref{eq:delta} and \eqref{eq:gamma}, with a sample of size \(N_{\textrm{MC}}\) and 
a truncated Poisson process at level \(K\).

\subsubsection{Results}
The result for \(K=100\) are given in Table~\ref{tab:surem_m100} and are considered as benchmark.
\begin{table}[ht!]
	\begin{center}
		\begin{tabular}{|c|c|c|c|}
			\hline
			&&&\tabularnewline
			\( \Psi(y)\) & \(\tilde{P}_\Psi(N_{\textrm{MC}},K)\) & \(\tilde{\Delta}_\Psi(N_{\textrm{MC}},K)\) & 
			\(\tilde{\Gamma}_\Psi(N_{\textrm{MC}},100)\)
			\tabularnewline&&&\tabularnewline
			\hline
			\(  y^2 \)   &  0.904526  (2.8e-5)  &  0.164247 (7.0e-5)   &  1.02012 (1.5e-4)    \tabularnewline
			\( \exp(-y) \)   &  1.36243 (3.3e-5)  &  -1.08837 (8.7e-5)  &  0.564459 (2.2e-4)   \tabularnewline
			\( \mathbbm{1}_{y>x} \)   &  0.47637 (1.1e-5)  &  -0.357681 (1.4e-5)   &  -0.0531064 (2.6e-5)  \tabularnewline
			\hline
		\end{tabular}
		% # temps execution 93261.4
		\caption{
			Results for the approximation \(X^{x,K}_T\) of the solution of 
			\eqref{eq:surem}. \(K=100\), 
			\(M=0.5\), \(x=0.04\), \(N_{\textrm{MC}}=2e9\).
			The program runs \(1.1e5\) seconds.}
		\label{tab:surem_m100}
	\end{center}
\end{table}

In Tables~\ref{tab:surem_m0}, \ref{tab:surem_m1}, \ref{tab:surem_m2}, we can see the
approximated biases for \(K=0, 1, 2\). We observe that according to Proposition~\ref{prop:controlerreur}, the bias decrease fast 
with \(K\) and the bias seems to be neglicted for \(K=2\),
even for the approximation of the derivatives.

Table~\ref{tab:lespk} gives  the empirical probability 
\(p^K\) to accept a Brownian bridge with the truncated algorithm
at level \(K\). It is obviously a monotonic function of \(K\).
We observe that \(p^2\approx p^{100}\) with a very large accuracy.

\begin{table}[ht!]
\begin{center}
\begin{tabular}{|c|c|c|c|}
	\hline
				&&&\tabularnewline
	\( \Psi(y)\) & \(\tilde{P}_\Psi(N_{\textrm{MC}},K)\)  &
	\(\tilde{\Delta}_\Psi(N_{\textrm{MC}},K) \)
	 & 
	 \(\tilde{\Gamma}_\Psi(N_{\textrm{MC}},K)\)
	\tabularnewline&
	\(\quad\quad- \tilde{P}_\Psi(N_{\textrm{MC}},100)\)
		&
\(\quad	\quad	- \tilde{\Delta}_\Psi(N_{\textrm{MC}},100)\) 
		&
\(\quad	\quad	 - \tilde{\Gamma}_\Psi(N_{\textrm{MC}},100)\) 
		\tabularnewline
	\hline
	\(  y^2 \)   &  
	1.66e-2
	(2.8e-5) 
	&  
	7.6e-4
	(6.9e-5)  
	&  
	9.7e-3
	(1.5e-4)  
	\tabularnewline
\( \exp(-y) \)   &  
1.2e-3
(3.2e-5)   
&  
-4.9e-2
(8.6e-5) 
&  
1.9e-2
(2.2e-4)  
\tabularnewline
\( \mathbbm{1}_{y > x} \)   &  
-8.0e-3
(1.1e-5)  
&  
-1.6e-2
(1.4e-5) 
&  
6.3e-3
(2.6e-5)  
 \tabularnewline
\hline
\end{tabular}
\caption{
		Errors with the truncated approximation \(X^{x,K}_T\) of the solution of 
		\eqref{eq:surem}. \(K=0\), \(M=0.5\), \(x=0.04\), \(N_{\textrm{MC}}=2e9\).
		The program runs \(1.1e4\) seconds. }
\label{tab:surem_m0}
\end{center}
\end{table}

\begin{table}[ht!]
\begin{center}
\begin{tabular}{|c|c|c|c|}
	\hline
			&&&\tabularnewline
	\( \Psi(y)\) & \(\tilde{P}_\Psi(N_{\textrm{MC}},K)\)  &
	\(\tilde{\Delta}_\Psi(N_{\textrm{MC}},K) \)
	 & 
	 \(\tilde{\Gamma}_\Psi(N_{\textrm{MC}},K)\)
	\tabularnewline&
	\(\quad\quad- \tilde{P}_\Psi(N_{\textrm{MC}},100)\)
		&
\(\quad	\quad	- \tilde{\Delta}_\Psi(N_{\textrm{MC}},100)\) 
		&
\(\quad	\quad	 - \tilde{\Gamma}_\Psi(N_{\textrm{MC}},100)\) 
		\tabularnewline
	\hline
	\(  y^2 \)   &  
	1.0e-4
	(2.8e-5) 
	&  
	4.1e-4
	(7.0e-5)  
	&  
	-9.5e-3
	(1.5e-4)    
	\tabularnewline
\( \exp(-y) \)   &  
1.0e-5
(3.3e-5)  
&  
1.4e-4
(8.7e-5)   
&  
-6.3e-3
(2.2e-4)  
\tabularnewline
\( \mathbbm{1}_{y>x} \)   &  
1.6e-5
(1.1e-5)  
&  
5.0e-6
(1.4e-5)  
&  
6.0e-4
(2.6e-5)  
\tabularnewline
\hline
\end{tabular}
\caption{	Results for the approximation \(X^{x,K}_T\) of the solution of 
	\eqref{eq:surem}. \(K=1\), \(M=0.5\), \(x=0.04\), \(N_{\textrm{MC}}=2e9\).
	The program runs \(1.2e4\) seconds.}
\label{tab:surem_m1}
\end{center}
\end{table}

\begin{table}[ht!]
\begin{center}
\begin{tabular}{|c|c|c|c|}
	\hline
				&&&\tabularnewline
	\( \Psi(y)\) & \(\tilde{P}_\Psi(N_{\textrm{MC}},K)\)  &
	\(\tilde{\Delta}_\Psi(N_{\textrm{MC}},K) \)
	 & 
	 \(\tilde{\Gamma}_\Psi(N_{\textrm{MC}},K)\)
	\tabularnewline&
	\(\quad\quad- \tilde{P}_\Psi(N_{\textrm{MC}},100)\)
		&
\(\quad	\quad	- \tilde{\Delta}_\Psi(N_{\textrm{MC}},100)\) 
		&
\(\quad	\quad	 - \tilde{\Gamma}_\Psi(N_{\textrm{MC}},100)\) 
		\tabularnewline
	\hline
	\(  y^2 \)   &   
	6e-6
	(2.8e-5)  
	&  
	-2.7e-5
	(7.0e-5)  
&  
	2.2e-4
	(1.5e-4) 
	\tabularnewline
\( \exp(-y) \)   &  
1.0e-5
(3.3e-5)  
&  
-1.0e-5
(8.7e-5)  
&  
4.2e-5
(2.2e-4) 
\tabularnewline
\(  \mathbbm{1}_{y>x} \)   &  
1.9e-5
(1.1e-5)  
&  
-9e-6
(1.4e-5)  
&  
-1.5e-5
(2.6e-5) 
\tabularnewline
\hline
\end{tabular}
\caption
{	Results for the approximation \(X^{x,K}_T\) of the solution of 
	\eqref{eq:surem}.\(K=2\),\(M=0.5\), \(x=0.04\), \(N_{\textrm{MC}}=2e9\).
	The program runs \(1.3e4\) seconds. }
\label{tab:surem_m2}
\end{center}
\end{table}

\begin{table}[h]
	\begin{center}
		\begin{tabular}{|c|c|c|c|c|}
			\hline
			\(K\)&0&1&2&100\\
			\hline
			\(p^K\)&0.877731 & 0.832898 & 0.832884&0.832877\\
			\hline
		\end{tabular}
		\caption{Probability \(p^K\) 
			defined in \eqref{eq:probaaccepttruncK} to accept a Brownian bridge path.}
		\label{tab:lespk}
	\end{center}
\end{table}

\subsection{CIR} \label{subsec:cir_an}
In this Section, we present the numerical results obtained for the simulation of the CIR process \((V_t,0\leq t \leq T)\), solution of \eqref{eq:CIR} (see Section~\ref{sec:cirmodel}). 
There is a large literature on the evaluation of \(\mathbb{E}(\Psi(V_T))\)
(see e.g.  \cite{livreAlfonsi} and references therein). Our aim is not to construct a specific 
algorithm for this particular case. Howerer, we think that it is relevant 
to illustrate the efficiency of our algorithm to this non trivial case.

The numerical experiments are computed with parameters \(\kappa = 0.5\), \(V_\infty = 0.04\), \(\varepsilon = 0.1\), \(T=1\) and the initial condition \(v = 0.04\). 
The algorithm differs from the two previous examples. We first apply
the Lamperti transform and simulate \(X_t^x = \eta(V_t^v)\) with our 
(almost) unbiased algorithm (with \(x = \eta(v)\)). 
Then, for any function \(\Psi\), we use the approximation
\(\tilde{P}_\Psi^X(x,N_{\textrm{MC}})\), \(\tilde{\Delta}_\Psi^X(x,N_{\textrm{MC}})\) and \(\tilde{\Gamma}_\Psi^X(x,N_{\textrm{MC}})\)
constructed for the process \(X^x\).
We deduce the corresponding approximation for the CIR 
\begin{align}
\bar{P}_\Psi^V(v,N_{\textrm{MC}}) &= \tilde{P}^X_\Psi(\eta(v),N_{\textrm{MC}})
\label{eq:prixbarcir}\\
\bar{\Delta}_\Psi^V(v,N_{\textrm{MC}}) &= \eta^\prime(v)\tilde{\Delta}_\Psi^X(\eta(v),N_{\textrm{MC}})
\label{eq:deltabarcir}\\
\bar{\Gamma}_\Psi^V(v,N_{\textrm{MC}}) &= \eta^\prime(v)^2\tilde{\Gamma}^X_\Psi(\eta(v),N_{\textrm{MC}}) + 
\eta^{\prime\prime}(v)\tilde{\Delta}_\Psi^X(\eta(v),N_{\textrm{MC}}).
\label{eq:gammabarcir}
\end{align}
\paragraph*{Description of the algorithms}
We first remind the quantities we aim to estimate.
Then, we describe the four algorithms we numerically compare in this Section.
\begin{enumerate}[1-] \setcounter{enumi}{-1}
	\item 
	The exact values are denoted  by \(P_\Psi(v)\), \(\Delta_\Psi(v)\) and \(\Gamma_\Psi(v)\), that is
\begin{align*}
P_\Psi(v) &= \mathbb{E}\Psi(V_T^v)\\
\Delta_\Psi(v) &= \dfrac{d}{dv}\mathbb{E}\Psi(V_T^v)\\
\Gamma_\Psi(v) &= \dfrac{d^2}{dv^2}\mathbb{E}\Psi(V_T^v).
\end{align*}
\item 
Our approximations \(\bar{P}_\Psi(N_{\textrm{MC}})\), \(\bar{\Delta}_\Psi(N_{\textrm{MC}})\) and \(\bar{\Gamma}_\Psi(N_{\textrm{MC}})\) are defined in \eqref{eq:prixbarcir}, \eqref{eq:deltabarcir} and \eqref{eq:gammabarcir}.
\item The approximations using an Euler scheme and finite difference approximation are denoted \(\hat{P}_\Psi(N_{\textrm{MC}},\delta)\),
\(\hat{\delta}_\Psi(N_{\textrm{MC}},\delta,dv)\) and \(\hat{\Gamma}_\Psi(N_{\textrm{MC}},\delta,dv)\) (see Sec.~\ref{sec:tuefd_vs_turem}).
  \item 
  We also approximate with an Euler scheme the expression of the derivatives obtained after the Malliavin integration by part (see Section~\ref{subsec:22} and \ref{subsec:23}): \(\check{\Delta}_\Psi(N_{\textrm{MC}},\delta)\)
  and \(\check{\Gamma}_\Psi(N_{\textrm{MC}},\delta)\).
\item Finally, we approximate \(\Delta_\Psi\) and \(\Gamma_\Psi\) thanks to the finite difference approximation applied to our unbiased estimators of
\(\bar{P}_\Psi^v(N_{\textrm{MC}})\), \(\bar{P}_\Psi^{v-dv}(N_{\textrm{MC}})\) and \(\bar{P}_\Psi^{v+dv}(N_{\textrm{MC}})\). We will denote  these approximations as \(\tilde{\Delta}_\Psi(N_{\textrm{MC}},dv)\) and \(\tilde{\Gamma}_\Psi(N_{\textrm{MC}},dv)\).
\end{enumerate}
The results and the corresponding standard deviations of these
estimators (with the truncated algorithm at level \(K=20\)) are given in Tables~\ref{tab:prixgrecquecir}, \ref{tab:prixgrecquecir_delta} and \ref{tab:prixgrecquecir_gamma}. We put in bold symbols the exact theoretical 
results when they are available.
For the function \(\Psi=\mathbbm{1}_{y>v}\), we have put in the reference column 
(\(P_\Psi\), \(\Delta_\Psi\), \(\Gamma_\Psi\)) the approximation
with our methods with a sample of size \(N_{\mathrm{MC}} = 1e12\).

\paragraph{Discussion on the results}
In any column, except the third one, we observe bias for the non smooth function \(\Psi(y)=\mathbbm{1}_{y>v}\). Moreover, the variance of
our algorithm is comparable to the variances of the biased one.
In a fixed time devoted for simulation, our unbiased algorithm
is always the most precise one in these examples.

\paragraph{Control of the error}
Even if the rigorous proof presented in Section~\ref{subsec:control_error}
can not be directly used for the CIR process, a similar control
of the error for the truncated algorithm should be obtained.
For \(K=20\) and the bounded function \(\Psi\) case (\(\Psi(y) = \mathbbm{1}_{\{y>v\}}\)), we obtained an accuracy of order \(1e-100\).

\begin{table}
\begin{center}
\begin{tabular}{|c|c|c|c|c|c|}
 \hline
\( \Psi(y)\) & \(P_\Psi\) &   \(\bar{P}_\Psi(N_{\textrm{MC}}^1) -P_\Psi\) & \( \hat{P}_\Psi(N^2_{\textrm{MC}}, \delta^1) - P_\Psi\) & 
\( \hat{P}(N_{\textrm{MC}}^2, \delta^2)- P_\Psi\)
 \tabularnewline
\hline
\( y \)  &  \textbf{0.04} &  -4.5e-9  (1.6e-8) 
 &  7e-7 (5.0e-7)
 &  2e-7 (5.1e-7) 
  \tabularnewline
\( \mathbbm{1}_{y>v} \) &   0.545628  &  0 (5.0e-7)  &  2.5e-4 (1.8e-5)  
&  3.1e-3 (1.6e-5)   
  \tabularnewline
\( \exp(-y) \) & \textbf{ 0.960910476 }  &  1e-9 (1.5e-8)  
 &  1.2e-7 (4.8e-7) 
 &  6e-6 (4.9e-7)
   \tabularnewline
\hline
\end{tabular}
 \caption{Estimation of the error (reference in bold) on the expectation and the corresponding
standard deviation for the CIR with different methods.
\(N^1_{\textrm{MC}}=1e12\), \(N^2_{\textrm{MC}}=1e9\), \(\delta^1=0.001\),
\(\delta^2=0.1\).
}
\label{tab:prixgrecquecir}
\end{center}
\end{table}

\begin{table}
\begin{center}
\begin{tabular}{|c|c|c|c|c|c|}
 \hline
\(\Psi\) & \(\Delta_\Psi\) & \( \bar{\Delta}_\Psi(N^1_{\textrm{MC}})\) &  
\(\hat{\Delta}_\Psi(N_{\textrm{MC}}^2,\delta^1)\)
& 
\(\check{\Delta}_\Psi(N_{\textrm{MC}}^3,\delta^2,dv)\)
& 
\(\tilde{\Delta}_\Psi(N_{\textrm{MC}}^3,dv)\)
 \tabularnewline
&&\(-\Delta_\Psi\)&
\(-\Delta_\Psi\)&
\(-\Delta_\Psi\)&
\(-\Delta_\Psi\) \tabularnewline
\hline  
\( y \)  &\textbf{0.606531}  &  1e-6 (6.5e-6) 
&  3.0e-3 (6.3e-5)  
&  7.8e-3 (5.1e-5)   
&  2e-6 (4.2e-5)    
\tabularnewline
 \( \mathbbm{1}_{y>v} \) & -15.3247626  &  0  (8.5e-5)  &  -4.3e-2 (7.4e-4)  
 &  -8.0e-1 (1.5e-3)    
 &  -0.32 (1.3e-3)  
 \tabularnewline
\( \exp(-y) \) &  \textbf{-0.58053743}  &  -2.1e-7 (1.3e-4)   
&  6.2e-2 (1.2e-3) 
&  7.6e-3 (4.9e-5) 
&  -2e-6 (4.0e-5) 
\tabularnewline
\hline
\end{tabular}
 \caption{Estimation of the first derivative and the corresponding
standard deviation for the CIR with different methods.
\(N_{\textrm{MC}}^1=1e12\),
\(N_{\textrm{MC}}^2=1e9\),
\(N_{\textrm{MC}}^3=1e10\),
\(\delta^1=0.001\),
\(\delta^2=0.1\),
\(dv=0.01\)}
\label{tab:prixgrecquecir_delta}
\end{center}
\end{table}

\begin{table}
\begin{tabular}{|c|c|c|c|c|c|}
	\hline
	\(\Psi\) & \(\Gamma_\Psi\) & \( \bar{\Gamma}_\Psi(N^1_{\textrm{MC}})\) &  
	%\( \hat{\Gamma}(1E9, \delta=0.001)\) 
	\(\hat{\Gamma}_\Psi(N_{\textrm{MC}}^2,\delta^1)\)
	& 
	\(\check{\Gamma}_\Psi(N_{\textrm{MC}}^3,\delta^2,dv)\)
	& 
	\(\tilde{\Gamma}_\Psi(N_{\textrm{MC}}^3,dv)\)
	\tabularnewline
	&&\(-\Gamma_\Psi\)&
	\(-\Gamma_\Psi\)&
	\(-\Gamma_\Psi\)&
	\(-\Gamma_\Psi\) \tabularnewline
	\hline  
	\( y \)  &
\textbf{ 0} & -5.1e-4 (1.9e-3)  
&  -4.0e-2 (4.0e-3)  
&  3.7e-4 (2.0e-2)   
&  -1.1e-2 (1.7e-2)     
\tabularnewline
 \( \mathbbm{1}_{y>v} \) &  91.0234  &   0  (2.3e-2)  &  3.0e-1 (5.5e-2)    
 &  12 (6.2e-1)    
 &  -7.1 (5.2e-1)  
 \tabularnewline
\(\exp(-y)\)&\textbf{ 0.35073}  &  -3.1e-3 (3.6e-2) 
&  -4.6e-1 (6.9e-2)  
&  -1.0e-2 (2.0e-2)   
&  1.0e-2 (1.6e-2)  
\tabularnewline
\hline
\end{tabular}
 \caption{Estimation of the second derivative  and the corresponding
standard deviation for the CIR with different methods.
\(N_{\textrm{MC}}^1=1e12\),
\(N_{\textrm{MC}}^2=1e9\),
\(N_{\textrm{MC}}^3=1e10\),
\(\delta^1=0.001\),
\(\delta^2=0.1\),
\(dv=0.01\)}
\label{tab:prixgrecquecir_gamma}
\end{table}

\section{Conclusion}
In this work, we  generalise the Beskos et al.~\cite{BPR06}
exact method to simulate the solution of one dimensional SDEs. We
simulate the Poisson process useful to reject the Brownian bridge paths
in a more efficient order (by increasing ordinates). It also allows
us to extend the methodology to more general drift functions \(\alpha\).
In this case, we introduce a new bias but we obtain a control of the error: it converges exponentially fast to \(0\) with the truncation parameter.

In addition, we  proposed to generalise the unbiased Monte Carlo algorithm to
the estimation of the derivatives \eqref{eq:delta} and \eqref{eq:gamma}.

In comparison with the previous classical numerical methods, our algorithm
is more efficient if we want to obtain a sufficiently good accuracy.
For rough approximations, the bias introduced by the Euler scheme has the same 
order  as the statistical error of our algorithm.

\def\refname{References}
\bibliographystyle{plain}
\bibliography{biblio}

\begin{minipage}{0.49\textwidth}
	\textbf{Victor Reutenauer}\\
	Fotonower \\
	30 rue Charlot\\ 
	F-75003 Paris\\
	\texttt{victor@fotonower.com}
\end{minipage}%
\begin{minipage}{0.49\textwidth}
	\textbf{Etienne Tanr\'e}\\
	Universit\'e C\^ote d'Azur, Inria, France.\\
	Team Tosca\\
	2004, route des Lucioles, BP 93\\
	F-06902 Sophia Antipolis Cedex\\
	\texttt{Etienne.Tanre@inria.fr}
\end{minipage}%
\begin{minipage}{0.33\textwidth}
\end{minipage}

\end{document}